\documentclass[reqno,a4paper, 11pt]{amsart}

\usepackage[a4paper=true,pdfpagelabels]{hyperref}
\usepackage{graphicx}

\usepackage[ansinew]{inputenc}
\usepackage{amsfonts,epsfig}
\usepackage{latexsym}
\usepackage{amsmath}
\usepackage{amssymb}

\newtheorem{theorem}{Theorem}[section]
\newtheorem{lemma}[theorem]{Lemma}
\newtheorem{corollary}[theorem]{Corollary}
\newtheorem{proposition}[theorem]{Proposition}

\newtheorem{lettertheorem}{Theorem}
\newtheorem{letterlemma}[lettertheorem]{Lemma}

\newtheorem{letterproposition}[lettertheorem]{Proposition}

\theoremstyle{definition}

\theoremstyle{remark}

\theoremstyle{remarks}

\numberwithin{equation}{section}

\newcommand{\B}{\mathcal{B}}
\newcommand{\D}{\mathbb{D}}
\newcommand{\DD}{\widehat{\mathcal{D}}}
\newcommand{\K}{\mathcal{K}}
\newcommand{\N}{\mathbb{N}}
\newcommand{\M}{\mathcal{M}}
\newcommand{\R}{\mathbb{R}}

\newcommand{\C}{\mathbb{C}}

\newcommand{\e}{\varepsilon}
\newcommand{\ep}{\varepsilon}

\renewcommand{\phi}{\varphi}

\newcommand{\T}{\mathbb{T}}

\def\VMOA{\mathord{\rm VMOA}}

\def\BRV{\mathord{\rm BRV}}

       \def\b{\beta}        
\def\d{\delta}           \def\e{\varepsilon}
     \def\om{\omega}      
       \def\t{\theta}       
                  \def\z{\zeta}
                  \def\vp{\varphi}

\def\R{{\mathcal R}}

\def\VMOA{\mathord{\rm VMOA}}
\def\BMOA{\mathord{\rm BMOA}}

\renewcommand{\H}{\mathcal{H}}

\newenvironment{Prf}{\noindent{\emph{Proof of}}}
{\hfill$\Box$ }
\begin{document}

\title[Integral operators mapping into $H^\infty$]{Integral operators mapping into the space of bounded analytic functions}


\thanks{This research was supported in part by
by Ministerio de Econom\'{\i}a y Competitivivad, Spain,  and the
European Union (FEDER) projects MTM2014-52865-P, MTM2011-26538, MTM2015-63699-P, and
MTM2012-37436-C02-01; La Junta de Andaluc{\'i}a, FQM133, FQM210 and
P09-FQM-4468; Academy of Finland project no. 268009, and the Faculty
of Science and Forestry of University of Eastern Finland project no.
930349.}

\date{\today}


\begin{abstract}
We address the problem of studying the boundedness, compactness and
weak compactness of the integral operators $T_g(f)(z)=\int_0^z
f(\z)g'(\z)\,d\z$ acting from a Banach space $X$ into $H^\infty$. We
obtain a collection of general results which are appropriately
applied
 and mixed with specific techniques in order to solve the posed
questions to particular choices of $X$.
\end{abstract}

\author[M. D. Contreras]{Manuel D. Contreras}

\address{Camino de los Descubrimientos, s/n\\
Departamento de Matem\'{a}tica Aplicada II and IMUS\\
Universidad de Sevilla\\
Sevilla, 41092\\
Spain.}\email{contreras@us.es}

\author[J.A. Pel\'aez]{Jos\'e A. Pel\'aez}
\address{Departamento de An\'alisis Matem\'atico\\
Facultad de Ciencias\\
29071, M\'alaga\\
Spain. } \email{japelaez@uma.es}

\author[Ch. Pommerenke]{Christian Pommerenke}
\address{Institut f\"{u}r Mathematik\\
Technische Universit\"{a}t\\
D-10623, Berlin\\
Germany. }
\email{pommeren@math.tu-berlin.de}

\author[J. R\"atty\"a]{Jouni R\"atty\"a}
\address{University of Eastern Finland, P. O. Box 111, 80101 Joensuu, Finland}
\email{jouni.rattya@uef.fi}

\maketitle

\tableofcontents

\section{Introduction}

Let $\H(\D)$ denote the space of analytic functions in the unit disc
$\D=\{z:|z|<1\}$, and let $\T$ stand for the boundary of $\D$. The
boundedness and compactness of the integral operators
    $$
    T_g(f)(z)=\int_0^z f(\zeta)g'(\zeta)\,d\zeta,\quad g\in\H(\D),
    $$
have been   extensively studied on different spaces of analytic
functions  since the seminal works by Aleman, Cima, Siskakis, and
the third author of this paper~\cite{AC,AS,Pom}. However, it seems
that little is known about $T_g$ mapping a Banach space $X\subset
\H(\D)$ into the space~$H^\infty$ of bounded analytic functions in
$\D$. As for this question, a good number of results were proved in
\cite{AJS} when $X=H^\infty$. Moreover, a classical result attributed to
Privalov~\cite{Privalov} shows that the Volterra operator
$V(f)(z)=\int_0^z f(\z)\,d\z$ is bounded from the Hardy space $H^1$
to the disc algebra $A$, the space of those $f\in\H(\D)$ which are
continuous on $\overline{\D}$. The main objective of this paper is to study the boundedness,
compactness and weak compactness of $T_g$ acting from $X$ into
$H^\infty$. To begin with, we will work in a wide framework
providing abstract approaches to these questions. Later on, these
general results will be applied to particular choices of $X$.

In order to explain  the first of our methods  some definitions are
needed. Let $\{\beta_n\}_{n=0}^\infty$ be a sequence of positive
numbers such that
$\lim_{n\to\infty}\sqrt[n]{\beta_n}= 1$, and
let $H(\beta)$ be the Hilbert space of analytic functions in $\D$
induced by the $H(\beta)$-pairing
    \begin{equation*}
    \langle f,g\rangle_{H(\beta)}=\lim_{r\to
    1^-}\sum_{n=0}^\infty\widehat{f}(n)\overline{\widehat{g}(n)}\beta_{n}r^{n},
    \end{equation*}
where $f(z)=\sum_{n=0}^\infty\widehat{f}(n)z^n$ and $g(z)=\sum_{n=0}^\infty\widehat{g}(n)z^n$. For $f\in \mathcal{H}(\D)$ and $0<r<1$,
define $f_r(z)=f(rz)$ for all $z\in\D$. Throughout the paper a Banach space
$X\subset\H(\D)$ is called admissible if it satisfies the following
conditions:
\begin{enumerate}
\item[(P1)] The disk algebra $A$ is contained in $X$;
\item[(P2)] If $f\in X$, then $f_r\in X$ and $\sup_{0<r<1}\|f_r\|_X\lesssim\|f\|_X$;
\item[(P3)] The point evaluation functional $\d_z$, $\delta_z(f)=f(z)$, belongs to $X^*$ for all $z\in\D$.
\end{enumerate}

\begin{theorem}\label{Thm-main-Xintro} 
Let $g\in\H(\D)$.
\begin{itemize}
\item[(i)]
Let $X,Y\subset\H(\D)$ be Banach spaces such that
$X^*\simeq Y$ via the $H(\beta)$-pairing. Then $T_g:X\to
H^\infty$ is bounded if and only if
$\sup_{z\in\D}\left\|G^{H(\beta)}_{g,z}\right\|_{Y}<\infty$, where
    \begin{equation*}
    \qquad\qquad\overline{G^{H(\beta)}_{g,z}(w)}=\int_{0}^zg'(\z)\overline{
    K^{H(\beta)}_{\z}(w)}\,d\z
    = \overline{T^*_g(K^{H(\beta)}_z)(w)}, \quad z\in\D,\quad
    w\in\overline{\D},
    \end{equation*}
and $K^{H(\beta)}_z$ denotes the reproducing kernel of the Hilbert space $H(\beta)$.
\item[(ii)] If $X$ is reflexive, admissible and $H^\infty\subset X$, then $T_g:\,X
\to H^\infty$ is bounded if and only if $T_g:\,X \to A$ is bounded.
\end{itemize}
\end{theorem}

The proof of Theorem~\ref{Thm-main-Xintro}(i) is simple and consists on rewriting $T_g(f)(z)$ appropriately. However, more effort is required to prove Theorem~\ref{Thm-main-Xintro}(ii) and we need to go through a previous study of the weak compactness of
integral operators acting on $H^\infty$. If  $X$ is separable and
$H^\infty\subset X$, we shall  prove that $g\in A$ whenever
$T_g:X\to H^\infty$ is bounded.
Since there exists $g\notin A$ such that $T_g:H^\infty\to
H^\infty$ is bounded \cite[Proposition~2.12]{AJS},
neither the
separability in the previous statement nor the reflexivity in
Theorem~\ref{Thm-main-Xintro}(ii) can be removed from the hypotheses.

The following result, whose
second part says that Privalov's theorem is in a sense sharp, is proved by
choosing $H(\beta)$ as the Hardy space $H^2$ in
Theorem~\ref{Thm-main-Xintro}.

\begin{theorem}\label{co:intro1} Let $g\in\H(\D)$.
\begin{itemize}
\item[\rm(i)] The following conditions are equivalent:
\begin{enumerate}
\item[\rm(a)] $T_g:H^\infty\to H^\infty$ is bounded;
\item[\rm(b)] $T_g:A\to H^\infty$ is bounded;
\item[\rm(c)] $\displaystyle \sup_{z\in\D}\left\|G^{H^2}_{g,z}\right\|_{\K}<\infty$.
\end{enumerate}
\item[\rm(ii)] $T_g:\mathcal{K}_a\to H^\infty$ is bounded if and only if $g$ is constant.
\end{itemize}
\end{theorem}

Here $\K$ denotes the space of the Cauchy
transforms and $\K_a$ stands for the Cauchy transforms induced by absolutely
continuous measures.

The election of $H(\beta)$ in Theorem~\ref{Thm-main-Xintro}  as the Bergman space $A^2_\om$ induced by a
regular weight (see section~\ref{sec:bergman} for the definition) is
a key  tool in order to prove the next result.

\begin{theorem}\label{Thm:g-is-constantintro}
 If $1<p<\infty$ and $\om$ is a regular weight, then the following conditions are
equivalent:
\begin{itemize}
\item[\rm(a)] $T_g:A^p_\om\to H^\infty$ is bounded if and only if $g$ is constant;
\item[\rm(b)] $T_g:A^p_\om\to A$ is bounded if and only if $g$ is constant;
\item[\rm(c)] $\displaystyle \int_0^1\frac{dt}{\left[(1-t)\om(t)\right]^{\frac1{p-1}}}=\infty$.
\end{itemize}
In particular, if $\alpha\geq p-2$ and $T_g:A^p_\alpha\to H^\infty$
is bounded, then $g$ is constant.
\end{theorem}

Let us observe that (a) and (b) are equivalent by Theorem~\ref{Thm-main-Xintro} and \cite[Corollary~7]{PelRatproj}.
The proof of (a)$\Leftrightarrow $(c) in
Theorem~\ref{Thm:g-is-constantintro} is based on a decomposition
technique on blocks which gives an equivalent norm in $A^p_\om$~\cite[Theorem~4]{PelRathg},
and on precise $L^p$-estimates of the reproducing kernels of the Dirichlet Hilbert space $\mathcal{D}_\om$~\cite[Theorem~4(ii)]{PR2015} induced by regular weights.

In fact, in the proof of the above theorem, it can be seen that if (c) is not satisfied, then $T_g:A^p_\om\to A$ is bounded for any polynomial $g$.

The condition $\sup_{z\in\D}\left\|G^{H(\beta)}_{g,z}\right\|_{Y}<\infty$ in
Theorem~\ref{Thm-main-Xintro}(i) might seem obscure, but it
turns out to be useful in praxis, for example in
Theorem~\ref{co:intro1}(ii) and
Theorem~\ref{Thm:g-is-constantintro}.

As for the question of the boundedness, we provide a second general approach that is closely related to bounded surjective projections, a topic intimately related to duality. Theorem~\ref{boundedness_Lp} below
allows us to describe those $g\in\H(\D)$ such that $T_g:X\to H^\infty$ is bounded when $X$ is an
$L^p$-quotient, in particular, when $X$ is the Bloch space or $\BMOA$.
We will also prove that the Hankel operator $H_{\bar{g}}$ and
the integral operator $T_g$ are not simultaneously bounded on
$H^\infty$, although it is well-known that this happens on $H^p$ for
$1<p<\infty$~\cite{AC,Pell}.

With regard to the weak compactness, we will focus on the cases $X=H^\infty$ and $X=A$. In fact, using a
result due to Bourgain \cite{Bourgain, Bourgain2}, which says that a
bounded operator $T:H^\infty\to X$ is weakly compact if and only
if it is completely continuous, and a result due to Lef\'evre
\cite[Theorem~1.2]{Lefevre} in the same spirit, we will prove the
following result.

\begin{theorem}\label{weakly_compact_imply_g_in_A}
Let $g\in\H(\D)$. If $T_g:H^\infty\to H^\infty$ is weakly compact,
then $g\in A$.
\end{theorem}

Theorem~\ref{weakly_compact_imply_g_in_A} is a key in our study
and in particular it is used to prove
Theorem~\ref{Thm-main-Xintro}(ii). The next result is also obtained
as a byproduct of Theorem~\ref{weakly_compact_imply_g_in_A}.

\begin{corollary}\label{co:weakAA}
Let $g\in\H(\D)$ such that $T_g:H^\infty\to H^\infty$ is bounded.
Then the following statements are equivalent:
\begin{enumerate}
\item[(a)]$T_g:H^\infty\to H^\infty$ is weakly compact;
\item[(b)]$T_g:H^\infty\to A$ is bounded;
\item[(c)]$T_g:A\to A$ is weakly compact.
\end{enumerate}
\end{corollary}

As for the compactness, an application of
Theorem~\ref{weakly_compact_imply_g_in_A} together with the use of
the $H(\beta)$-pairing provides the following result.

\begin{theorem}\label{th:compactH2pairing} Let $g\in\H(\D)$ and $X\subset\H(\D)$ admissible such that $H^\infty\subset X$.
\begin{enumerate}
\item[\rm(i)] Then $T_g:X\to H^\infty$ is compact if and only if $T_g:X\to A$ is compact.
\item[\rm(ii)] If $Y\subset\H(\D)$ satisfies $Y\simeq X^*$ via the
$H(\beta)$-pairing and for every bounded sequence in $X$, there is a
subsequence which converges uniformly on compact subsets of $\D$ to
an element of $X$, then $T_g:X\to A$ is compact if and only
if $\{T_g^*(K^{H(\beta)}_z):z\in \D\}=
\{G^{H(\b)}_{g,z}:z\in \D\} $ is relatively compact in $Y$.
\end{enumerate}
\end{theorem}

We shall use Theorem~\ref{th:compactH2pairing}(i) to
prove the following result, which is in stark contrast with the fact
that there is a $g\in \H(\D)$ such that $T_g:H^\infty\to H^\infty$
is bounded but $T_g$ is unbounded acting on $A$.

\begin{theorem}\label{compactness_Hinfty} Let $g\in\H(\D)$. Then the following statements are equivalent:
\begin{enumerate}
\item[(i)] $T_g:H^\infty\to H^\infty$ is compact;
\item[(ii)] $T_g:A\to A$ is compact;
\item[(iii)] $T_g:H^\infty\to A$ is compact;
\item[(iv)] $T_g:A\to H^\infty$ is compact;
\item[(v)] $\{G^{H^2}_{g,z}:z\in\D\}$ is relatively compact in $\K$.
\end{enumerate}
\end{theorem}

\medskip\par Furthermore, we will combine
Theorem~\ref{th:compactH2pairing} with Littlewood-Paley formulas for
the $H^2$-paring and the $A^2_\om$-pairing in order to obtain
concrete descriptions of the $g\in \mathcal H(\D)$ such that $T_g: X\to
H^\infty$ is compact. In particular, we will prove the following
result.

\begin{theorem}\label{th:BMOAcompactintro}
Let $g\in\H(\D)$.
\begin{itemize}
\item[\rm(a)] Each polynomial $g$ induces a bounded operator $T_g:
H^1\to H^\infty$ but only constant functions induce compact
integral operators from $H^1$ to $H^\infty$.
\item[\rm(b)] If $T_g:\BMOA\to H^\infty$ is bounded, then the following statements
are equivalent:
\begin{itemize}
\item[(i)] $T_g:\BMOA\to H^\infty$ is compact;
\item[(ii)] $T_g:\VMOA\to A$ is compact;
\item[(iii)] $T_g:\BMOA\to A$ is compact;
\item[(iv)] $T_g:\VMOA\to H^\infty$ is compact;
\item[(v)] $\{ G^{H^2}_{g,z}:z\in \D\}$ is a relatively compact in $H^1$;
\item[(vi)] $
    \displaystyle\lim_{R\to1^-}\sup_{z\in\D}
    \int_{\T}\left(\int_{\Gamma(\xi)\setminus D(0,R)}\left|(G^{H^2}_{g,z})'(w)\right|^2\,dA(w)\right)^{\frac{1}{2}}\,|d\xi|=0,
    $
    where
$$\Gamma(\xi)=\left\{\z\in\D:|\t-\arg\z|<\frac12\left(1-|\z|\right)\right\}$$ is the  lens-type region with vertex at $\xi=e^{i\t}$.
\end{itemize}

\end{itemize}
\end{theorem}
We shall use the theory of tent spaces introduced by Coifman, Meyer
and Stein \cite{CMS} in the proof of
Theorem~\ref{th:BMOAcompactintro}(b). In particular, a clever stopping-time argument \cite[(4.3)]{CMS} will be employed to prove
that (vi) implies (i).

Finally, let us recall that each $g\in H^\infty$ (resp. $g\in A$) with
bounded radial variation induces a bounded integral operator $T_g$
on $H^\infty$ (resp. on $A$). However, as far as we know, the
reverses are open problems. We will show that each $g\in A$ with bounded radial variation induces a
bounded integral operator from $H^\infty$ into the disc algebra, see Proposition~\ref{pr:BRVA} below.

The rest of the paper is organized as follows.
Section~\ref{section:boundedness} is devoted to the study of
boundedness, and the equivalence (a)$\Leftrightarrow$(c) of
Theorem~\ref{Thm:g-is-constantintro} is proved there.
Theorem~\ref{Thm-main-Xintro}(i) coincides with
Theorem~\ref{Thm-main-X}, and Theorem~\ref{co:intro1} is contained
in Theorem~\ref{Thm:main-1}. In
Section~\ref{section:weakcompactness} we prove
Theorem~\ref{weakly_compact_imply_g_in_A}, Corollary~\ref{co:weakAA}
and Corollary~\ref{le:XAinfty}, which contains
Theorem~\ref{Thm-main-Xintro}(ii). Section~\ref{section:compactness}
contains the proofs of Theorems~\ref{th:compactH2pairing},
\ref{compactness_Hinfty}, and \ref{th:BMOAcompactintro}.
Section~\ref{section:comments} contains a further discussion about
some of our results and in particular (see
Theorem~\ref{H2Hinftypositive} below) shows that condition
$\sup_{z\in\D}\left\|G^{H^2}_{g,z}\right\|_{H^2}<\infty$ can be
neatly rewritten when $g$ has nonnegative Taylor coefficients.

\section{Bounded integral operators mapping into
$H^\infty$}\label{section:boundedness}

\subsection{General results}

We begin with introducing some notation and a couple of results
which will be repeatedly used throughout the paper. Let $T(X,H^\infty)$ denote the set of $g\in\mathcal{H}(\D)$ such that $T_g:X\to H^\infty$ is
bounded, and let $T_c(X,H^\infty)$ denote the symbols $g$ for which $T_g:X\to H^\infty$ is compact. For $f\in\H(\D)$ and $0<r<1$, let
    \begin{equation*}
    \begin{split}
    M_p(r,f)&=\left(\frac{1}{2\pi}\int_{0}^{2\pi} |f(re^{i\t})|^p\,d\t\right)^{1/p},\quad
    0<p<\infty,\\
    M_\infty(r,f)&=\max_{|z|=r}|f(z)|.
    \end{split}
    \end{equation*}
For $0<p\le \infty $, the Hardy space $H^p$  consists of $f\in
\H(\mathbb D)$ such that $\|f\|_{H^p}=
\sup_{0<r<1}M_p(r,f)<\infty$. The proof of the following known result is standard and hence omitted.

\begin{letterlemma}\label{polynomials} Let $X\subset\H(\D)$ be an admissible Banach space. Then the following assertions hold:
\begin{enumerate}
\item $T(X,H^\infty)$ endowed with the norm $\|g\|_*=\|T_g\|_{X\to H^\infty}+|g(0)|$ is a Banach space.
\item If the multiplication operator $M_z:X\to X$ is bounded and $T_z:X\to H^\infty$
is bounded, then $T(X,H^\infty)$ contains all polynomials. If in addition $T_z$ is compact, then $T_c(X,H^\infty)$ contains all polynomial.
\end{enumerate}
\end{letterlemma}

Let $\overline{\mathcal{P}}^X$ denote the closure of polynomials in
$X$.

\begin{proposition}\label{bd-from-X=from-polynomials}
Let $X\subset\H(\D)$ be an admissible Banach space and $g\in\H(\D)$. Then the following statements are equivalent:
\begin{enumerate}
\item[\rm(a)] $T_g:X\to H^\infty$ is bounded;
\item[\rm(b)] $T_g:\overline{\mathcal{P}}^X\to H^\infty$ is bounded.
\end{enumerate}
Moreover,
    $$
    \|T_g\|_{\overline{\mathcal{P}}^X\to H^\infty}\asymp\|T_g\|_{X\to H^\infty}.
    $$
\end{proposition}

\begin{proof}
It is clear that (a) implies (b), and $\|T_g\|_{\overline{\mathcal{P}}^X\to H^\infty}\le\|T_g\|_{X\to
H^\infty}$. Now, assume (b), and let $f\in X$. Since $f_r\in\overline{\mathcal{P}}^X$ for each $0<r<1$
by (P1), (P2) yields
    $$
    \|T_g(f_r)\|_{H^\infty}
    \le\|T_g\|_{\overline{\mathcal{P}}^X\to H^\infty}\|f_r\|_X
    \lesssim \|T_g\|_{\overline{\mathcal{P}}^X\to H^\infty}\|f\|_X,\quad 0<r<1.
    $$
Since $T_g(f)(z)=\lim_{r\to 1^-}T_g(f_r)(z)$ for all $z\in\D$, we
deduce $|T_g(f)(z)|\lesssim\|T_g\|_{\overline{\mathcal{P}}^X\to H^\infty}
\|f\|_X$. This yields (a) and $\|T_g\|_{X\to
H^\infty}\lesssim\|T_g\|_{\overline{\mathcal{P}}^X\to H^\infty}$.
\end{proof}

Let $\{\beta_n\}_{n=0}^\infty$ be a sequence of positive numbers such that $\lim_{n\to\infty}\sqrt[n]{\beta_n}= 1$. Then
$F_\beta(z)=\sum_{n=0}^\infty \frac{z^n}{\beta_n}\in\H(\D)$, and
since H\"older's inequality gives
    \begin{equation*}
    |f(z)|\le \sum_{n=0}^\infty |\widehat{f}(n)||z|^n\le
    \|f\|_{H(\beta)}F_\beta(|z|^2),\quad z\in\D,
    \end{equation*}
the norm convergence in $H(\beta)$ implies the uniform convergence in compact subsets of $\D$. In particular,
the point evaluation functionals are bounded on $H(\beta)$, and therefore the Riesz representation theorem guarantees the existence
of reproducing kernels $K^{H(\beta)}_z\in H(\b)$ such that
    \begin{equation}\label{eq:reproducingbeta}
    f(z)=\langle f,K^{H(\beta)}_z\rangle_{H(\beta)},\quad f\in
H(\beta),\quad z\in\D.
    \end{equation}
The identity $K^{H(\beta)}_z(\zeta)=
F_\beta(\zeta\overline{z})=\sum_{n=0}^\infty
\frac{(\zeta\overline{z})^n}{\beta_n}$ follows by using the standard orthonormal basis of normalized monomials.

Two normed spaces $X,Y\subset\H(\D)$ satisfy the duality relation
$X^*\simeq Y$ via the $H(\beta)$-pairing if $\langle
f,g\rangle_{H(\beta)}$ exists for all $f\in X$ and $g\in Y$, the
linear functional $L_g(f)=\langle
    f,g\rangle_{H(\beta)}$ on $X$ satisfies
    $
    C^{-1}\|g\|_{Y}\le\|L_g\|\le C\|g\|_{Y}
    $
for some constant $C>0$, and each $L\in X^*$ equals to $L_g$ for some $g\in Y$.
Therefore, if $X^*\simeq Y$ via the $H(\beta)$-pairing, the
anti-linear mapping $g\mapsto L_g$ is a bijection from $Y$ to $X^*$ and the norms of $g$ and $L_g$ are comparable.

With these preparations we can state and prove the first characterization of the bounded integral operators.

\begin{theorem}\label{Thm-main-X}
Let $g\in\H(\D)$ and $X,Y\subset\H(\D)$ Banach spaces such that
$X^*\simeq Y$ via the $H(\beta)$-pairing. Then $T_g:X\to
H^\infty$ is bounded if and only if
$\sup_{z\in\D}\left\|G^{H(\beta)}_{g,z}\right\|_{Y}<\infty$, where
    \begin{equation}\label{eq:Gbeta}
    \overline{G^{H(\beta)}_{g,z}(w)}=\int_{0}^zg'(\z)\overline{
    K^{H(\beta)}_{\z}(w)}\,d\z
    = \overline{T^*_g(K^{H(\beta)}_z)(w)}, \quad z\in\D,\quad
    w\in\overline{\D}.
    \end{equation}
Moreover,
    \begin{equation}\label{eq:normGbeta}
    \|T_g\|_{X\to
    H^\infty}\asymp\sup_{z\in\D}\left\|G^{H(\beta)}_{g,z}\right\|_{Y}.
    \end{equation}
\end{theorem}

\begin{proof}
For each $z\in\D$,
    $$
    \overline{G^{H(\beta)}_{g,z}(w)}=\int_{0}^zg'(\z)\overline{
    K^{H(\beta)}_\z(w)}\,d\z=\sum_{n=0}^\infty \left(\int_0^z
    \frac{\z^n}{\beta_n} g'(\z)\,d\z\right) \overline{w}^n,
    $$
and hence
    \begin{equation}
    \begin{split}\label{le:tgformula}
    T_g(f)(z) &=\int_0^z g'(\z)f(\z)\,d\z= \lim_{r\to 1^-}\int_0^z g'(\z)f(r\z)\,d\z\\
    &=\lim_{r\to 1^-} \int_0^z \left(\sum_{n=0}^\infty \widehat{f}(n)\z^n r^n \right)g'(\z)\,d\z\\
    &=\lim_{r\to 1^-} \sum_{n=0}^\infty \widehat{f}(n) \left(\int_0^z\frac{\z^n}{\beta_n} g'(\z)\,d\z\right)\beta_n r^n
    =\langle f,G^{H(\beta)}_{g,z}\rangle_{H(\beta)},\quad z\in\D,
    \end{split}
    \end{equation}
for all $f\in\H(\D)$. Therefore $T_g:X\to H^\infty$ is bounded if
and only if each $G^{H(\beta)}_{g,z}$ induces via
$\cdot\mapsto\langle\cdot,G^{H(\beta)}_{g,z}\rangle_{H(\beta)}$ a bounded
linear functional in $X$ with norm uniformly bounded in~$z$. Thus
$T_g:X\to H^\infty$ is bounded if and only if $\sup_{z\in\D}\left\|G^{H(\beta)}_{g,z}\right\|_{Y}<\infty$, and further, \eqref{eq:normGbeta} is satisfied. The second
equality in \eqref{eq:Gbeta} follows by \eqref{eq:reproducingbeta} and \eqref{le:tgformula}:
$$\langle T_g(f),K^{H(\beta)}_{z}\rangle_{H(\beta)}=T_g(f)(z)  =\langle f,G^{H(\beta)}_{g,z}\rangle_{H(\beta)},\quad z\in\D.
$$
Ergo, $T_g^*(K^{H(\beta)}_{z})=G^{H(\beta)}_{g,z}$ for all $z\in \D$.
\end{proof}

By using Proposition~\ref{bd-from-X=from-polynomials} and
Theorem~\ref{Thm-main-X} we obtain the following result.

\begin{proposition}\label{Thm-main-XA}
Let $X\subset\H(\D)$ be admissible and $Y\subset\H(\D)$ such that $X^*\simeq Y$
via the $H(\beta)$-pairing. Then the following statements are
equivalent:
\begin{enumerate}
\item[\rm(a)] $T_g:\overline{\mathcal{P}}^X\to A$ is bounded;
\item[\rm(b)] $g\in A$ and $\sup_{z\in\D}\left\|G^{H(\beta)}_{g,z}\right\|_{Y}<\infty$.
\end{enumerate}
\end{proposition}

\begin{proof}
Proposition~\ref{bd-from-X=from-polynomials} and
Theorem~\ref{Thm-main-X} show that (a) implies (b). Conversely,
assume that $g\in A$ and
$\sup_{z\in\D}\left\|G^{H(\beta)}_{g,z}\right\|_{Y}<\infty$. If $p$
is a polynomial, then $T_g(p)$ belongs to $A$. If
$f\in\overline{\mathcal{P}}^X$, there exists a sequence of
polynomials $\{p_n\}$ such that $\|p_n-f\|_X\to0$. By
Theorem~\ref{Thm-main-X}, $T_g:X\to H^\infty$ is bounded, so
$\|T_g(p_n)-T_g(f)\|_{H^\infty}\lesssim\|p_n-f\|_X\to 0$. Since
$T_g(p_n)\in A$ for all $n$,
 then  $T_g(f)\in A$.
\end{proof}

Theorem~\ref{Thm-main-X} is based on characterizations of dual
spaces. The following  approach to the question of when $T_g:X\to
H^\infty$ is bounded consists of using surjective integral operators
involving a kernel. Since descriptions of dual spaces is closely related to bounded projections, in some concrete  applications Theorems~\ref{Thm-main-X} and~\ref{boundedness_Lp} yield the same conclusion.

For $1<p<\infty$, write $p'$ for the conjugate of $p$, $\frac{1}{p}+\frac{1}{p'}=1$. If $p=1$, set $p'=\infty$.

\begin{theorem}\label{boundedness_Lp}
Let $X\subset\H(\D)$ be a Banach space, $(\Omega,\mu)$ a finite
measure space and $\mu$ a positive measure. Let
$K:\D\times\Omega\to\C$ be a kernel such that
\begin{enumerate}
\item[\rm(i)] $z\mapsto K(z,w)$ is analytic for all $w\in\Omega$;
\item[\rm(ii)] $w\mapsto K(z,w)$ is measurable for all $z\in\D$;
\item[\rm(iii)] $\sup_{z\in E,
w\in\Omega}|K(z,w)|<\infty$ for each compact subset $E\subset\D$.
\end{enumerate}
Let $g\in\H(\D)$, $1<p\le\infty$ and
    $$
    P(h)(z)=\int_\Omega h(w)K(z,w)\, d\mu(w),\quad z\in\D,
    $$
such that $P:L^p(\Omega)\to X$ is bounded and onto. Then $T_g:X\to
H^\infty$ is bounded if and only if
    \begin{equation}\label{Eq:Lp}
    \sup_{z\in\D}\int_\Omega\left|\int_0^zg'(\z)K(\z,w)\,d\z\right|^{p'}\,d\mu(w)<\infty.
    \end{equation}
Moreover, there exists $M=M(P)>0$ such that
    \begin{equation}
    \begin{split}\label{eq:proy2}
  M^{-p'} \|T_g\|^{p'}_{X\to H^\infty} \le  \sup_{z\in\D}\int_\Omega\left|\int_0^zg'(\z)K(\z,w)\,d\z\right|^{p'}\,d\mu(w)\le  \|P\| ^{p'}  \|T_g\|^{p'}_{X\to H^\infty}.
  \end{split}
    \end{equation}
\end{theorem}

\begin{proof} Let first $p=\infty$, and assume that $T_g:X\to H^\infty$ is bounded. Then
    \begin{equation*}
    \begin{split}
    |T_g(P(h))(z)|
    \le\|T_g\|_{X\to H^\infty}\|P(h)\|_X
    \le\|T_g\|_{X\to H^\infty}\|P\|_{L^\infty\to X}\|h\|_{L^\infty},\quad z\in\D,
    \end{split}
    \end{equation*}
for all $h\in L^\infty=L^\infty(\Omega)$. Fix $z\in\D$ and, for
each $w\in\Omega$, define
    $$
    h_z(w)=
    \left\{
    \begin{array}{ll}
    \frac{\overline{\int_0^zg'(\z)K(\z,w)\,d\z}}{\left|\int_0^zg'(\z)K(\z,w)\,d\z\right|},\quad &\textrm{if }\int_0^zg'(\z)K(\z,w)\,d\z\neq 0,\\
    0, &\textrm{if } \int_0^zg'(\z)K(\z,w)\,d\z=0,
    \end{array}\right.
    $$
so that $h_z\in L^\infty$ for each $z\in\D$ with
$\sup_{z\in\D}\| h_z\|_{L^\infty}=1$. Then Fubini's theorem gives
    $$
    \int_\Omega\left|\int_0^zg'(\z)K(\z,w)\,d\z\right|\,d\mu(w)
    =|T_g(P(h_z))(z)|\le\|T_g\|_{X\to H^\infty}\|P\|_{L^\infty\to X}<\infty,\quad z\in\D,
    $$
so \eqref{Eq:Lp} and the second inequality in \eqref{eq:proy2}
follow.

Conversely, since $P:L^\infty\to X$ is bounded and onto, the open
mapping theorem ensures the existence of a constant $M=M(P)>0$ such
that, for each $f\in X$ there is $h\in L^\infty$ such that $f=P(h)$
and $\|h\|_{L^\infty}\le M\|f\|_X$. A reasoning similar to that
above now yields
    \begin{equation}\label{kulli2}
    \begin{split}
    |T_g(f)(z)|&=\left|\int_0^zg'(\z)\left(\int_\Omega h(w)K(\z,w)\,d\mu(w)\right)d\z\right|\\
    &\le\int_\Omega|h(w)|\left|\int_0^zg'(\z)K(\z,w)\,d\z\right|\,d\mu(w)\\
    &\le M\|f\|_X\int_\Omega\left|\int_0^zg'(\z)K(\z,w)\,d\z\right|\,d\mu(w),
    \end{split}
    \end{equation}
and thus $T_g: X \to H^\infty$ is bounded with
    $$
    M^{-1} \|T_g\|_{X\to H^\infty}
    \le\sup_{z\in\D}\int_\Omega\left|\int_0^zg'(\z)K(\z,w)\,d\z\right|\,d\mu(w).
    $$

Let now $1<p<\infty$, and denote $L^p=L^p(\Omega)$ for short. Assume
that $T_g:X\to H^\infty$ is bounded. Fubini's theorem gives
    \begin{equation*}
    \begin{split}
    \int_\Omega\left|\int_0^zg'(\z)K(\z,w)\,d\z\right|^{p'}\,d\mu(w)
    &=\sup_{\|h\|_{L^p}\le1}\left|\int_0^zg'(\z)P(h)(\z)\,d\z\right|^{p'}\\
    &=\sup_{\|h\|_{L^p}\le1}\left|T_g(P(h))(z)\right|^{p'}\le\|T_g\|_{X\to H^\infty}^{p'}\|P\|_{L^p\to
    X}^{p'},
    \end{split}
    \end{equation*}
so \eqref{Eq:Lp} and the second inequality in \eqref{eq:proy2}
follows.

Conversely, if \eqref{Eq:Lp} is satisfied, then a reasoning similar
to that in \eqref{kulli2}, involving the open mapping theorem,
Fubini's theorem and H\"older's inequality, shows that $T_g:X\to
H^\infty$ is bounded and the claimed norm estimate is satisfied.
\end{proof}

\subsection{Hardy and related spaces}

In this section, Theorem~\ref{Thm-main-X} and
Proposition~\ref{Thm-main-XA} together with
Theorem~\ref{boundedness_Lp} are applied to spaces of analytic
functions  with dual spaces described in terms of the $H^2$-pairing.
To do this, some notation is needed. The space $\BMOA$ (resp.
$\VMOA$) consists of functions in the Hardy space $H^1$ having
bounded (resp. vanishing) mean oscillation on the boundary $\T$. The
space $\mathcal K$ of Cauchy transforms consists of analytic
functions in $\D$ of the form
    $$
    (K\mu)(z)=\int_{\T} \frac{d\mu(\xi)}{1-\overline \xi z},
    $$
where $\mu$ belongs to the space of finite complex Borel
measures on $\T$, and is endowed with the norm
$\|f\|_{\K}=\inf\{\|\mu\|:f=K\mu\}$~\cite{CiMaRo}. Here $\|\mu\|$
denotes the total variation of~$\mu$. For $w\in\T$, the function
$f_w(z)=\frac{1}{1-\overline w z}$ belongs to $\mathcal K$ and
$\|f_w\|_{\mathcal K}=1$. Moreover, $\K$ admits the decomposition
$\K=\K_a+\K_s$, where $\K_a$ and $\K_s$ denote the Cauchy transforms
induced by absolute continuous and singular measures, respectively.
Since $A^*\simeq\K$ and $K_a^*\simeq H^\infty$ via the
$H^2$-pairing~\cite[p.~89--91]{CiMaRo}, this decomposition yields

\begin{equation}\label{eq:A^**}
A^{**}=H^\infty \oplus \K_s^*.
\end{equation}
Recall that
    \begin{equation*}
    \overline{G^{H^2}_{g,z}(w)}=\int_{0}^z\frac{g'(\z)}{1-\overline{w}\z}\,d\z, \quad z\in\D,\quad
    w\in\overline{\D},
    \end{equation*}
and for $g\in\H(\D)$ and $z\in\D$, define
    $$
    d\mu_{g,z}(\z)=\left(\int_0^z\frac{g'(u)}{\z-u}\,du\right)\frac{d\z}{2\pi i},\quad \z\in\T.
    $$
Theorem~\ref{co:intro1} is contained in the
following result.

\begin{theorem}\label{Thm:main-1}
Let $g\in\H(\D)$.
\begin{itemize}
\item[\rm(i)] $T_g:A\to A$ is bounded if and only if $g\in A$ and $\sup_{z\in\D}\left\|G^{H^2}_{g,z}\right\|_{\K}<\infty$. In particular, if $g\in A$ and $\sup_{z\in\D}\|\overline{\mu_{g,z}}\|<\infty$, then $T_g:A\to A$ is bounded.
\item[\rm(ii)] The following statements are equivalent:
\begin{enumerate}
\item[\rm(a)] $T_g:H^\infty\to H^\infty$ is bounded;
\item[\rm(b)] $T_g:A\to H^\infty$ is bounded;
\item[\rm(c)] $\displaystyle\sup_{z\in\D}\left\|G^{H^2}_{g,z}\right\|_{\K}<\infty$.
\end{enumerate}
Moreover, $\|T_g\|_{H^\infty\to H^\infty}=\|T_g\|_{A\to H^\infty}$. In particular, if $ \sup_{z\in\D}\|\overline{\mu_{g,z}}\|<\infty$, then $T_g:H^\infty\to H^\infty$ is bounded.
\item[\rm(iii)] If $1<p<\infty$, then $T_g:H^p\to H^\infty$ is bounded if and only if $\sup_{z\in\D}\left\|G^{H^2}_{g,z}\right\|_{H^{p'}}<\infty$.
\item[\rm(iv)] $T_g:H^1\to H^\infty$ is bounded if and only if $\sup_{z\in\D}\left\|G^{H^2}_{g,z}\right\|_{\BMOA}<\infty$. In particular, $T(H^1,H^\infty)$ contains all polynomials.
\item[\rm(v)] The following statements are equivalent:
\begin{enumerate}
\item[\rm(a)] $T_g:\BMOA\to H^\infty$ is bounded;
\item[\rm(b)] $T_g:\VMOA\to H^\infty$ is bounded;
\item[\rm(c)] $\displaystyle\sup_{z\in\D}\left\| G^{H^2}_{g,z}\right\|_{H^1}<\infty$.
\end{enumerate}
In particular, if $T_g:\VMOA\to H^\infty$ is bounded, then $T_g:H^\infty\to H^\infty$ is bounded.
\item[\rm(vi)] $T_g:\mathcal{K}_a\to H^\infty$ is bounded if and only if $g$ is constant. In particular, if $0<p<1$, then $T_g:H^p\to H^\infty$ is bounded only if $g$ is constant.
\end{itemize}
\end{theorem}

\begin{proof}
(i) The first assertion is a consequence of
Proposition~\ref{Thm-main-XA} and \cite[Theorem~4.2.2]{CiMaRo}. To see the second assertion, it suffices to use
Fubini's theorem and the Cauchy integral formula to show that
$\overline{\mu_{g,z}}$ is one of the representing measures of
$G^{H^2}_{g,z}$ as a function in~$\K$, that is,
    \begin{equation*}
  \overline{G^{H^2}_{g,z}(w)}=\overline{(K\mu_{g,z})
  (w)}= \int_\T\frac{d\mu_{g,z}(\z)}{1-\overline{w}\z},\quad w\in\D.
    \end{equation*}

(ii) The fact that (a) and (b) are equivalent is due to
Proposition~\ref{bd-from-X=from-polynomials}, and its proof shows
that $\|T_g\|_{H^\infty\to H^\infty}=\|T_g\|_{A\to H^\infty}$.
Theorem~\ref{Thm-main-X} gives the equivalence between~(b) and~(c).

(iii) This is a consequence of the duality $(H^p)^*\simeq H^{p'}$
for $1<p<\infty$, and Theorem~\ref{Thm-main-X}. An alternative way
to deduce the assertion is to note that the Szeg\"o projection
    $$
    P(f)(z)=\frac{1}{2\pi}\int_0^\pi \frac{f(e^{i\theta})}{1-ze^{-i\theta}}\, d\theta,\quad z\in\D,
    $$
is bounded from $L^p(\T)$ onto $H^p$~\cite[Theorem~9.6]{Zhu},
and
then apply Theorem~\ref{boundedness_Lp}.

(iv) Fefferman's duality theorem states that
$(H^1)^*\simeq\BMOA$ via the $H^2$-pairing \cite[Theorem
9.20]{Zhu}. Therefore the assertion follows by
Theorem~\ref{Thm-main-X}. The function $z\mapsto\log(1-z)$ belongs
to $\BMOA$, and hence
$\sup_{z\in\D}\|G^{H^2}_{g,z}\|_{\BMOA}<\infty$ for $g(z)=z$.
Therefore $T(H^1,H^\infty)$ contains all polynomials by
Lemma~\ref{polynomials}.

(v) Since $\VMOA$ is the closure of polynomials in $\BMOA$, the
equivalence between (a) and (b) is due to
Proposition~\ref{bd-from-X=from-polynomials}. Using again
Theorem~\ref{Thm-main-X} and the duality relation $\VMOA^*\simeq
H^1$~\cite[Theorem 9.28]{Zhu},
 we deduce that (b) and (c) are equivalent.
Alternatively, the Szeg\"o projection is bounded from
$L^\infty(\T)$ onto $\BMOA$ by~\cite[Theorem~9.21]{Zhu},
and hence (a) and (c) are equivalent by
Theorem~\ref{boundedness_Lp}.

(vi) If $T_g:\mathcal{K}_a\to H^\infty$ is bounded, then
Theorem~\ref{Thm-main-X} gives
$\sup_{z\in\D}\|G^{H^2}_{g,z}\|_{H^\infty}<\infty$. Assume that $g$
is not constant and let $N\in\N\cup\{0\}$ such that
$\widehat{g}(N+1)\ne0$. Since
    \begin{equation*}
    \begin{split}
    \int_0^z\frac{g'(\z)}{1-\overline{w}\z}\,d\z
    =\sum_{k=0}^\infty\left(\sum_{n=0}^\infty (n+1)\widehat{g}(n+1)\frac{z^{n+k+1}}{n+k+1}\right)\overline{w}^k,
    \end{split}
    \end{equation*}
we deduce
    \begin{equation*}
    \begin{split}
    \sup_{z\in\D}\|G^{H^2}_{g,z}\|_{H^\infty}
    &\ge\sup_{z\in\D}\left|\sum_{k=0}^\infty\left(\sum_{n=0}^\infty(n+1)\widehat{g}(n+1) \frac{z^{n+1}}{n+k+1}\right)|z|^{2k}\right|\\
    &= \sup_{z\in\D}\left| \sum_{n=0}^\infty(n+1)\widehat{g}(n+1) \left(\sum_{k=0}^\infty \frac{|z|^{2k} }{n+k+1}\right) z^{n+1}\right|
    \\ & \ge \sup_{0<r<1}\left(\frac{1}{2\pi}\int_0^{2\pi}
    \left| \sum_{n=0}^\infty(n+1)\widehat{g}(n+1) \left(\sum_{k=0}^\infty \frac{r^{2k} }{n+k+1}\right) (re^{i\theta})^{n+1}\right|^2
     \,d\theta \right)^{1/2}
     \\ & = \sup_{0<r<1}\left( \sum_{n=0}^\infty(n+1)^2|\widehat{g}(n+1)|^2 \left(\sum_{k=0}^\infty \frac{r^{2k+n+1} }{n+k+1}\right)^2
    \right)^{1/2}
     \\ & \ge \sup_{0<r<1} (N+1)|\widehat{g}(N+1)|\sum_{k=0}^\infty \frac{r^{2k+N+1} }{N+k+1}
     \\ & \ge \lim_{r\to 1^-} |\widehat{g}(N+1)|r^{N-1}\sum_{k=0}^\infty \frac{r^{2(k+1)} }{k+1}
     \\ & \ge \lim_{r\to 1^-}
     |\widehat{g}(N+1)|r^{N-1}\log\frac{1}{1-r^2}=\infty.
    \end{split}
    \end{equation*}
This implies that $T_g:\K_a\to H^\infty$ is bounded only if $g$ is
constant. The second assertion follows by \eqref{eq:inclusion} below.
\end{proof}

Theorem~\ref{Thm:main-1}(vi) implies that $T(X,H^\infty)$ contains
constant functions only when $\mathcal{K}_a\subset X\subset \H(\D)$.
The chain of inclusions
    \begin{equation}\label{eq:inclusion}
    \K_a\subset \K\subset H^{1,\infty}\subset \cap_{0<p<1}H^p\subset
    \cap_{0<p<1}A^p_\om
    \end{equation}
gives examples of such $X$. Here $H^{1,\infty}$ denotes the weak
Hardy space of order $1$~\cite[p.~35]{CiMaRo} and $\om$ is any
radial weight. For the embedding $\K\subset H^{1,\infty}$,
see~\cite[p.~75]{CiMaRo}. Moreover,
\cite[Proposition~4.1.17]{CiMaRo} says that
    $
    \K\subset H_v^\infty=\{f\in \H(\D):\sup_{z\in \D} v(|z|)|f(z)|\}
    $
for $v(r)=1-r$. Therefore $T_g$ maps the previous space into
$H^\infty$ if and only if $g$ is constant.

\subsection{Weighted Bergman and related
spaces}\label{sec:bergman}

In this section
we deal with spaces of analytic functions in $\D$ which dual spaces
are usually described in terms of the $A^2_\om$-pairing. To do this,
some notation are in order. A nonnegative integrable function $\om$
in~$\D$ is called a weight. It is radial if $\omega(z)=\omega(|z|)$
for all $z\in\D$. For $0<p<\infty$ and a weight $\omega$, the
weighted Bergman space $A^p_\omega$ consists of $f\in\H(\D)$ such
that
    $$
    \|f\|_{A^p_\omega}^p=\int_\D|f(z)|^p\omega(z)\,dA(z)<\infty,
    $$
where $dA(z)=\frac{dx\,dy}{\pi}$ is the normalized Lebesgue area
measure on $\D$. That is, $A^p_\om=L^p_\om\cap \H(\D)$ where
$L^p_\om$ is the corresponding Lebesgue space. As usual,
write~$A^p_\alpha$ for the standard weighted Bergman space induced
by the radial weight $(1-|z|^2)^\alpha$.

We will write $\DD$ for the class of radial weights such that
$\widehat{\om}(z)=\int_{|z|}^1\om(s)\,ds$ is doubling, that is,
there exists $C=C(\om)\ge1$ such that $\widehat{\om}(r)\le
C\widehat{\om}(\frac{1+r}{2})$ for all $0\le r<1$. We refer to
\cite[Lemma~2.1]{PRAntequera} for equivalent descriptions of the
class $\DD$.
 A radial weight
$\om$ is called regular, denoted by $\om\in\R$, if $\om\in\DD$ and
$\om(r)(1-r)\asymp\widehat{\om}(r)$ for all $0\le r<1$. As to
concrete examples, every standard weight as well as those given in
\cite[(4.4)--(4.6)]{AS} are regular.

The reproducing kernels of the Bergman space~$A^2_\om$ induced by $\om\in\DD$
are given by
    $$
    B^\om_\z(z)=\sum_{n=0}^\infty\frac{(\overline{\z}z)^n}{2\om_{2n+1}},\quad \z,\,z\in\D,
    $$
where $\om_{2n+1}=\int_{0}^1 r^{2n+1}\om(r)\,dr$ for all $n$. The
orthogonal Bergman projection $P_\om$ from $L^2_\om$ to $A^2_\om$
can be written as
    \begin{equation*}\label{intoper}
    P_\om(f)(z)=\int_{\D}f(\z)\overline{B^\om_{z}(\z)}\,\om(\z)dA(\z),\quad z\in\D.
    \end{equation*}
If $\om\in\R$, then $P_\om$ maps $L^p_\om$ onto $A^p_\om$ for
$1<p<\infty$, and $L^\infty(\D)$ into the Bloch
space~\cite[Theorem~5]{PelRatproj}. Recall that the Bloch space $\B$
consists of $f\in\H(\D)$ such that
$\|f\|_\B=\sup_{z\in\D}|f'(z)|(1-|z|^2)+|f(0)|<\infty$. The little
Bloch space $\B_0$ is the closure of polynomials in $\B$. For
$g\in\H(\D)$ and $z\in\D$, write
    $$
    \overline{G^{A^2_\om}_{g,z}(w)}=\int_{0}^zg'(\z)\overline{B_\zeta^\om(w)}\,d\z, \quad w\in\overline{\D}.
    $$

\begin{theorem}\label{Thm-main-2}
Let $\om\in\R$ and $g\in\H(\D)$.
\begin{itemize}
\item[\rm(i)] If $1<p<\infty$, then $T_g:A^p_\om\to H^\infty$ is bounded if and only if $\sup_{z\in\D}\left\|G^{A^2_\om}_{g,z}\right\|_{A^{p'}_\om}<\infty$.
\item[\rm(ii)] If $0<p\le1$, then $T_g:A^p_\om\to H^\infty$ is bounded if and only if $g$ is constant.
\item[\rm(iii)] The following statements are equivalent:
\begin{enumerate}
\item[\rm(a)] $T_g:\B\to H^\infty$ is bounded;
\item[\rm(b)] $T_g:\B_0\to H^\infty$ is bounded;
\item[\rm(c)] $\sup_{z\in\D}\left\|G^{A^2_\om}_{g,z}\right\|_{A^1_\om}<\infty$.
\end{enumerate}
\end{itemize}
\end{theorem}

\begin{proof}
(i) This follows either by Theorem~\ref{Thm-main-X} and
\cite[Corollary~7]{PelRatproj}, or Theorem~\ref{boundedness_Lp} and
\cite[Theorem~5]{PelRatproj}.

(ii) By Theorem~\ref{Thm:main-1}(vi) and \eqref{eq:inclusion} it
suffices to show that $\cap_{0<p<1}H^p \subset A^1_\om$. The
continuous weight $\widetilde{\om}(r)=\frac{\widehat{\om}(r)}{1-r}$
belongs to $\R$ and $\om\asymp \widetilde{\om}$.
 Since each $f\in \cap_{0<p<1}H^p$ satisfies
$M_1(r,f)\lesssim(1-r)^{-\e}$, $0<r<1$, for any $\ep>0$ (see
\cite[Theorem 5.9]{Duren1970}), we have
    \begin{equation*}
    \|f\|_{A^1_\om}\asymp\|f\|_{A^1_{\widetilde{\om}}}\lesssim\int_0^1(1-r)^{-\e}\widetilde{\om}(r)\,dr.
    \end{equation*}
Since $(1-r)^{-\e}\widetilde{\om}(r)$ is a weight for sufficiently small $\e=\e(\widetilde{\om})>0$ by \cite[p.~10]{PelRat}, we deduce $\cap_{0<p<1}H^p \subset A^1_\om$.

(iii) Proposition~\ref{bd-from-X=from-polynomials} shows that (a)
and (b) are equivalent. Let us prove  the equivalence between (a)
and (c).
 The Bergman projection $P_\om$
maps $L^\infty (\D)$ into the Bloch space $\B$ by
\cite[Theorem~5]{PelRatproj}. Moreover, for given $f\in\B$, the
function
    $$
    g(z)=\widehat{f}(0)+ \frac{\widehat{\om}(z)}{|z|\om(z)}\sum_{n=0}^\infty (2n+1)\widehat{f}(n)z^n
    =\widehat{f}(0)+ \frac{\widehat{\om}(z)}{|z|\om(z)}\left(2zf'(z)+f(z)-f(0) \right),
    $$
belongs to $L^\infty(\D)$, and satisfies $P_\om(g)=f$. Thus $P_\om:L^\infty (\D)\to\B$ is onto, and therefore $T_g:\B\to H^\infty$ is bounded if and only if $\sup_{z\in\D}\left\|G^{A^2_\om}_{g,z}\right\|_{A^1_\om}<\infty$ by Theorem~\ref{boundedness_Lp}.
\end{proof}

Next, we use  Theorem~\ref{Thm-main-2}(i)  to prove the equivalence
between (a) and (c) in Theorem~\ref{Thm:g-is-constantintro}.
 Some notation and preliminary results are now in order. If
$g(z)=\sum_{k=0}^\infty \widehat{g}(k) z^k$ is analytic in $\D$ and
$n_1,n_2\in\N\cup\{0\}$, write
    $$
    S_{n_1,n_2}g(z)=\sum_{k=n_1}^{n_2-1}\widehat{g}(k)z^k,\quad n_1<n_2.
    $$

The following result is essentially known~\cite{PelRathg}, but we
include a proof for the sake of completeness.

\begin{lemma}\label{le:Deltan}
Let $1<p,C_1, C_2<\infty$. Then
    $$
    \left\|S_{n_1,n_2}\left(\frac{1}{1-z}\right)\right\|_{H^p}\asymp n_2^{(1-\frac{1}{p})}
    $$
    for all $n_1,n_2\in\N$ such that $C_1\le\frac{n_2}{n_1}\le C_2$.
\end{lemma}

\begin{proof}
Let $h(z)=(1-z)^{-1}$. By \cite[Lemma~10]{PelRathg} and the M. Riesz
projection theorem,
    \begin{equation}\label{eq:j11}
    \begin{split}
    \left\|S_{n_1,n_2}h\right\|^p_{H^p}
    \asymp M^p_p\left(1-\frac{1}{n_2},S_{n_1,n_2}h\right)
    \lesssim M^p_p\left(1-\frac{1}{n_2},h\right)\asymp n_2^{p-1}.
    \end{split}
    \end{equation}
On the other hand, the hypothesis $1<C_1\le\frac{n_2}{n_1}\le C_2<\infty$ yields
    \begin{equation*}
    \begin{split}
    M_\infty\left(1-\frac{1}{n_1}, S_{n_1,n_2}h\right)\asymp n_2-n_1\asymp n_2.
    \end{split}
    \end{equation*}
Furthermore, by using the well-know inequality $M_\infty(r,f)\lesssim (\rho-r)^{-\frac1p}M_p(\rho,f)$, $0<r<\rho<1$, we deduce
    \begin{equation*}
    \begin{split}
    M_\infty\left(1-\frac{1}{n_1}, S_{n_1,n_2}h\right)
    \lesssim \left(\frac{1}{n_1}-\frac{1}{n_2}\right)^{-1/p}
    M_p\left(1-\frac{1}{n_2}, S_{n_1,n_2}h\right)
    \asymp n_2^{1/p}\left\|S_{n_1,n_2}h\right\|_{H^p},
    \end{split}
    \end{equation*}
and hence $n_2^{p-1}\lesssim \left\|S_{n_1,n_2}h\right\|^p_{H^p}$. This together with
\eqref{eq:j11} proves the assertion.
\end{proof}

Throughout the rest of the section we assume, without loss of
generality, that \linebreak[4] $\int_0^1 \om(s)\,ds=1$. For each
$n\in\N\cup\{0\}$, define $r_n=r_n(\om)\in[0,1)$ by
    \begin{equation}\label{rn}
    \widehat{\om}(r_n)=\int_{r_n}^1 \om(s)\,ds=\frac{1}{2^n}.
    \end{equation}
Clearly, $\{r_n\}_{n=0}^\infty$ is a non-decreasing sequence of
distinct points on $[0,1)$ such that $r_0=0$ and $r_n\to1^-$, as
$n\to\infty$. For $x\in[0,\infty)$, let $E(x)$ denote the integer
such that $E(x)\le x<E(x)+1$, and set
$M_n=E\left(\frac{1}{1-r_{n}}\right)$. Write
    $$
    I(0)=I_{\om}(0)=\left\{k\in\N\cup\{0\}:k<M_1\right\}
    $$
and
   \begin{equation*}
    I(n)=I_{\om}(n)=\left\{k\in\N:M_n\le
    k<M_{n+1}\right\}
  \end{equation*}
for all $n\in\N$. If
$f(z)=\sum_{n=0}^\infty\widehat{f}(n)z^n$ is analytic in~$\D,$ define the
polynomials $\Delta^{\om}_nf$ by
    \[
    \Delta_n^{\om}f(z)=\sum_{k\in I_{\om}(n)}\widehat{f}(n)z^k,\quad n\in\N\cup\{0\}.
    \]

The following equivalent norm in $A^p_\om$ plays an important role
in the proof of Theorem~\ref{Thm:g-is-constantintro}.

\begin{lettertheorem}\label{th:dec}
Let $1<p<\infty$ and $\om\in\DD$ such that $\int_0^1\om(t)\,dt=1$.
Then
    $$
    \|f\|_{A^p_\om}^p\asymp\sum_{n=0}^\infty2^{-n}\|\Delta_n^\om f\|_{H^p}^p
    $$
for all $f\in\H(\D)$.
\end{lettertheorem}

Theorem~\ref{th:dec} follows by \cite[Proposition~11]{PRAntequera},
see also the proof of \cite[Theorem~4]{PelRathg}. With these
preparations we are ready to prove the equivalence between (a) and
(c) in Theorem~\ref{Thm:g-is-constantintro}.

\medskip

\begin{Prf}{\em{Theorem~\ref{Thm:g-is-constantintro}}}.
The equivalence between (a) and (b)
 follows from
Corollary~\ref{le:XAinfty} below.  Suppose that (a) is satisfied and
assume on the contrary to the assertion that
    \begin{equation}\label{Eq:A-H-g-polynomial-bounded-condition}
    \int_0^1\frac{dr}{\widehat{\om}(r)^\frac1{p-1}}<\infty.
    \end{equation}
If $g(z)=z$, then Theorem~\ref{Thm-main-2}(i) states that $g\in T(A^p_\om,H^\infty)$
if and only if
    \begin{equation*}
    \begin{split}\label{eq:bi1}
    \sup_{z\in\D}\int_\D\left|\int_0^z\overline{B^\om_{\z}(w)}\,d\z\right|^{p'}\om(w)\,dA(w)<\infty,
    \end{split}
    \end{equation*}
where
    $$
    \int_0^z\overline{B^\om_{\z}(w)}\,d\z
    =\sum_{n=0}^\infty\frac{\overline{w}^n}{2\om_{2n+1}}\int_0^z\z^n\,d\z
    =\frac{z}{\overline{w}}(K^\om_w)'(z)
    $$
and $K^\om_w$ denotes the reproducing kernel of the Hilbert space
$\mathcal{D}_\om$ \cite[(2.8)]{PR2015}. Recall that~$\mathcal{D}_\om$ consists of $f\in\H(\D)$ such that
    $$
    \|f\|^2_{\mathcal{D}_\om}=\int_\D |f'(z)|^2\,\om(z)\,dA(z)+|f(0)|^2\om(\D)<\infty.
    $$
Consequently, \cite[Theorem~4(ii)]{PR2015} yields
    $$
    \left\| \int_0^z\overline{B^\om_{\z}(w)}\,d\z \right\|_{A^{p'}_v}^{p'}\asymp\int_0^{|z|}\frac{\widehat{v}(t)}{\widehat{\om}(t)^{p'}}\,dt,\quad |z|\to1^-,
    $$
for all $\om,v\in\DD$. In particular,
    $$
    \sup_{z\in\D}\int_\D\left|\int_0^z\overline{B^\om_{\z}(w)}\,d\z\right|^{p'}\om(w)\,dA(w)\asymp\int_0^1\frac{dr}{\widehat{\om}(r)^\frac1{p-1}}.
    $$
This contradicts (a), and thus we have shown that (a) implies (c).
In particular, by using these estimates and Lemma~\ref{polynomials},
we deduce that if $1<p<\infty$ and $\om\in\R$ such that
\eqref{Eq:A-H-g-polynomial-bounded-condition} is satisfied, then
$T(A^p_\om,H^\infty)$ contains all polynomials.

Assume now (c). To see that (a) is satisfied, by
Theorem~\ref{Thm-main-2}(i), it suffices to show that
    \begin{equation*}\label{eq:infty}
    \sup_{z\in\D}\int_\D\left|\int_0^zg'(\z)\overline{B^\om_{\z}(w)}\,d\z\right|^{p'}\om(w)\,dA(w)=\infty
    \end{equation*}
for each non-constant $g\in\H(\D)$. For $g(z)=\sum_{k=0}^\infty \widehat{g}(k)z^k\in\H(\D)$,
    \begin{equation*}
    \begin{split}
    g'(\z)\overline{B^\om_\z(w)}
    &=\left(\sum_{n=0}^\infty (n+1)\widehat{g}(n+1)\z^n\right)\left(\sum_{k=0}^\infty \frac{(\z\overline{w})^k}{2\om_{2k+1}}\right)\\
    &=\sum_{j=0}^\infty \left(\sum_{k=0}^j\frac{\overline{w}^k}{2\om_{2k+1}}(j-k+1)\widehat{g}(j-k+1)\right)\z^j,
    \end{split}
    \end{equation*}
and hence
    \begin{equation}
    \begin{split}\label{eq:bi2}
    \int_{0}^z g'(\z)\overline{B^\om_\z(w)}\,d\z
    &=\sum_{j=0}^\infty \left(\sum_{k=0}^j\frac{\overline{w}^k}{2\om_{2k+1}}(j-k+1)\widehat{g}(j-k+1)\right)\frac{z^{j+1}}{j+1}\\
    &=\sum_{k=0}^\infty \left(\sum_{j=k}^\infty\frac{j-k+1}{j+1}\widehat{g}(j-k+1)z^{j+1}\right) \frac{\overline{w}^k}{2\om_{2k+1}}\\
    &=\sum_{k=0}^\infty \left(\sum_{n=0}^\infty\frac{n+1}{n+k+1}\widehat{g}(n+1)z^{n+k+1}\right) \frac{\overline{w}^k}{2\om_{2k+1}}.
    \end{split}
    \end{equation}
By using \eqref{eq:bi2}, Theorem~\ref{th:dec} ,
\cite[Lemma~E]{PelRathg}(a) with $\lambda=\{\om^{-1}_{2k+1}\}$ and then
\cite[Lemma~E]{PelRathg}(b) with $\lambda=\{|z|^k\}$, we deduce
    \begin{equation*}
    \begin{split}
    &\int_\D\left|\int_0^zg'(\z)\overline{B^\om_{\z}(w)}\,d\z\right|^{p'}\om(w)\,dA(w)\\
    &=\int_\D\left|\sum_{k=0}^\infty \left(\sum_{n=0}^\infty\frac{n+1}{{n+k+1}}\overline{\widehat{g}(n+1)}\overline{z}^{n+k+1}\right) \frac{w^k}{2\om_{2k+1}}\right|^{p'}\om(w)\,dA(w)\\
    &\asymp\sum_{j=0}^\infty 2^{-j}\left\|\sum_{k\in I_\om(j)}\left(\sum_{n=0}^\infty\frac{n+1}{{n+k+1}}\overline{\widehat{g}(n+1)}\overline{z}^{n+k+1}\right) \frac{w^k}{2\om_{2k+1}}\right\|^{p'}_{H^{p'}}\\
    &\gtrsim\sum_{j=0}^\infty\frac{2^{-j}}{(\om_{2M_j+1})^{p'}}
    \left\|\sum_{k\in I_\om(j)}\left(\sum_{n=0}^\infty\frac{n+1}{{n+k+1}}\overline{\widehat{g}(n+1)}\overline{z}^{n+1}\right)(\overline{z}w)^k \right\|^{p'}_{H^{p'}}\\
    &=\sum_{j=0}^\infty\frac{2^{-j}}{(\om_{2M_j+1})^{p'}}
    \left\|\sum_{k\in I_\om(j)}\left(\sum_{n=0}^\infty\frac{n+1}{{n+k+1}}\overline{\widehat{g}(n+1)}\overline{z}^{n+1}\right) (|z|w)^k \right\|^{p'}_{H^{p'}}\\
    &\gtrsim\sum_{j=0}^\infty\frac{2^{-j}}{(\om_{2M_j+1})^{p'}}
    \left\|\sum_{k\in I_\om(j)}\left(\sum_{n=0}^\infty\frac{n+1}{{n+k+1}}\overline{\widehat{g}(n+1)}\overline{z}^{n+1}\right) w^k \right\|^{p'}_{H^{p'}}|z|^{M_{j+1}p'},
    \end{split}
    \end{equation*}
and hence Fubini's theorem yields
    \begin{equation*}
    \begin{split}
    &\sup_{z\in\D}\int_\D\left|\int_0^zg'(\z)\overline{B^\om_{\z}(w)}\,d\z\right|^{p'}\om(w)\,dA(w)\\
    &\gtrsim\sup_{z\in\D}\sum_{j=0}^\infty\frac{2^{-j}}{(\om_{2M_j+1})^{p'}}
    \left\|\sum_{k\in I_\om(j)}\left(\sum_{n=0}^\infty\frac{n+1}{{n+k+1}}\overline{\widehat{g}(n+1)}\overline{z}^{n+1}\right) w^k \right\|^{p'}_{H^{p'}}|z|^{M_{j+1}p'}\\
    &=\sup_{0<r<1}\sup_{t\in[-\pi,\pi)}
    \sum_{j=0}^\infty\frac{2^{-j}}{(\om_{2M_j+1})^{p'}}
    \left\|\sum_{k\in I_\om(j)}\left(\sum_{n=0}^\infty\frac{n+1}{{n+k+1}}\overline{\widehat{g}(n+1)}(re^{it})^{n+1}\right) w^k \right\|^{p'}_{H^{p'}}r^{M_{j+1}p'}\\
    &\ge\sup_{0<r<1}\frac{1}{2\pi}\int_{-\pi}^\pi\left(
    \sum_{j=0}^\infty\frac{2^{-j}}{(\om_{2M_j+1})^{p'}}
    \left\|\sum_{k\in I_\om(j)}\left(\sum_{n=0}^\infty\frac{n+1}{{n+k+1}}\overline{\widehat{g}(n+1)}(re^{it})^{n+1}\right) w^k \right\|^{p'}_{H^{p'}}r^{M_{j+1}p'}\right)dt\\
    &=\sup_{0<r<1}\sum_{j=0}^\infty\frac{2^{-j}}{(\om_{2M_j+1})^{p'}}\left(\frac{1}{(2\pi)^2}\int_{-\pi}^\pi
    \int_{-\pi}^\pi\left|\sum_{k\in I_\om(j)}\left(\sum_{n=0}^\infty\frac{n+1}{{n+k+1}}\overline{\widehat{g}(n+1)}(re^{it})^{n+1}\right) e^{ik\theta} \right|^{p'}\,dt\,d\theta\right)r^{M_{j+1}p'}\\
    &=\sup_{0<r<1}
    \sum_{j=0}^\infty \frac{2^{-j}}{(\om_{2M_j+1})^{p'}}\left(\frac{1}{(2\pi)^2}\int_{-\pi}^\pi
    \int_{-\pi}^\pi\left|\Gamma_\theta(re^{it}) \right|^{p'}\,dt\,d\theta\right)r^{M_{j+1}p'},
    \end{split}
    \end{equation*}
where $\Gamma_\theta(re^{it})=\sum_{n=0}^\infty(n+1)\overline{\widehat{g}(n+1)} \left( \sum_{k\in I_\om(j)}\frac{e^{ik\theta} }{n+k+1}\right) (re^{it})^{n+1}$. Since $g$ is non-constant, there exists $N\in\N\cup\{0\}$ such
that $|\widehat{g}(N+1)|>0$. Hence the M.~Riesz projection theorem
(or H\"older's inequality and the Cauchy integral formula) yields
    \begin{equation*}
    \begin{split}
    &\sup_{z\in\D}\int_\D\left|\int_0^zg'(\z)\overline{B^\om_{\z}(w)}\,d\z\right|^{p'}\om(w)\,dA(w)\\
    &\gtrsim(N+1)^{p'}|\widehat{g}(N+1)|^{p'}\sup_{0<r<1}r^{(N+1)p'}\sum_{j=0}^\infty\frac{2^{-j}r^{M_{j+1}p'}}{(\om_{2M_j+1})^{p'}}\frac{1}{2\pi}\int_{-\pi}^\pi
     \left|\sum_{k\in I_\om(j)}\frac{e^{ik\theta} }{N+k+1}\right|^{p'}\,d\theta.
    \end{split}
    \end{equation*}
By using \cite[Lemma~E]{PelRathg}(b) and Lemma~\ref{le:Deltan},
with $n_1=M_j$ and $n_2=M_{j+1}$, together with \cite[Lemma~6]{PelRathg}, we deduce
    \begin{equation*}
    \begin{split}
    &\frac{1}{2\pi}\int_{-\pi}^\pi
    \left| \sum_{k\in I_\om(j)}\frac{e^{ik\theta} }{N+k+1}\right|^{p'}\,d\theta
    \gtrsim\frac{1}{(N+M_{j+1})^{p'}}\frac{1}{2\pi}\int_{-\pi}^\pi
    \left|\sum_{k\in I_\om(j)}e^{ik\theta} \right|^{p'}\,d\theta\\
    &\qquad =\frac{1}{(N+M_{j+1})^{p'}}\left\|\Delta^\om_j\left(\frac{1}{1-z}\right)\right\|^{p'}_{H^{p'}}
    \asymp\frac{M_{j+1}^{p'-1}}{(N+M_{j+1})^{p'}}\asymp
    M_{j+1}^{-1}.
    \end{split}
    \end{equation*}
Consequently,
    \begin{equation*}
    \begin{split}
    \sup_{z\in\D}\int_\D\left|\int_0^zg'(\z)\overline{B^\om_{\z}(w)}\,d\z\right|^{p'}\om(w)\,dA(w)
    \gtrsim(N+1)^{p'}|\widehat{g}(N+1)|^{p'}\sum_{j=0}^\infty\frac{2^{-j}}{(\om_{2M_j+1})^{p'}M_{j+1}}.
    \end{split}
    \end{equation*}
Now \cite[Lemma~6]{PelRathg}, (see also \cite[Lemma~1]{PRAntequera})
and \eqref{rn} yield
    \begin{equation*}
    \begin{split}
    \infty&=\int_0^1\frac{dt}{\widehat{\om}(t)^{p'-1}}
    =\sum_{j=0}^\infty\int_{1-\frac1{M_j}}^{1-\frac1{M_{j+1}}}\frac{dt}{\widehat{\om}(t)^{p'-1}}
    \le\sum_{j=0}^\infty\frac{1}{\widehat{\om}\left(1-\frac{1}{M_{j+1}}\right)^{p'-1}}\left(\frac1{M_j}-\frac1{M_{j+1}}\right)\\
    &\asymp\sum_{j=0}^\infty\frac{1}{\widehat{\om}\left(1-\frac{1}{M_{j+1}}\right)^{p'-1}M_{j+1}}
    \asymp\sum_{j=0}^\infty\frac{2^{-j}}{(\om_{2M_j+1})^{p'}M_{j+1}},
    \end{split}
    \end{equation*}
and we deduce that for each non-constant analytic symbol $g$, the operator $T_g$ does not map $A^p_\om$ into
 $H^\infty$ if  $\int_0^1(\widehat{\om}(t))^{1-p'}dt$ diverges.
\end{Prf}

\subsection{The space $T(H^\infty,H^\infty)$}

The aim of this section is to provide further information related to
the space of symbols $g$ that induce bounded operators $T_g:H^\infty\to H^\infty$. Recall that
$X\subset\H(\D)$ endowed with a seminorm $\rho$ is called
conformally invariant or M\"{o}bius invariant if there exists a
constant $C>0$ such that
    $$
    \sup_{\vp\in\M}\rho(g\circ\varphi)\le C\rho(g),\quad g\in X,
    $$
where $\M$ is the set of all M\"{o}bius transformations $\vp$ of
$\D$ onto itself. Classical examples of conformally invariant spaces
are BMOA, the Bloch space $\mathcal B$ and $H^\infty$.

\begin{proposition} The Banach space $T(H^\infty, H^\infty)$
equipped with the norm $\|g\|_*=\|T_g\|_{(H^\infty,
H^\infty)}+|g(0)|$ has the following properties:
\begin{enumerate}
\item[(i)]
 Polynomials are not dense in $T(H^\infty,
H^\infty)$;
\item[(ii)] There exists $g\in T(H^\infty, H^\infty)$ such that $g_r\not\to g$ in the norm of $T(H^\infty,
H^\infty)$;
\item[(iii)] $T(H^\infty, H^\infty)$ endowed with the seminorm
    $\|T_g\|_{(H^\infty,
H^\infty)}$ is conformally invariant. Indeed,
    $$
    \sup_{\vp\in\M}\|T_{g\circ\vp}\|_{H^\infty\to
H^\infty}\le 2\| T_g\|_{H^\infty\to H^\infty}.
    $$
    \end{enumerate}
\end{proposition}

\begin{proof} (i) Assume that polynomials are dense in $T(H^\infty, H^\infty)$.
By \cite[Proposition~2.12]{AJS}, there exists $g\in T(H^\infty, H^\infty)\setminus A$. Let $\{p_n\}$ be a sequence of
polynomials such that $\| T_{g-p_n}\|_{H^\infty\to H^\infty}\to0$,
as $n\to\infty$. Since $\| g-p_n\|_{H^\infty}\le C \|
T_{g-p_n}\|_{H^\infty\to H^\infty}$ and $A$ is closed in $H^\infty$,
we deduce $g\in A$. This leads to  a contradiction and therefore
polynomials are not dense in $T(H^\infty, H^\infty)$.

(ii) The argument employed in (i) gives the assertion.

(iii) Take $g\in T(H^\infty,H^\infty)$, and for each $a\in\D$ let
$\vp_a(z)=\frac{a-z}{1-\overline{a}z}$. Then, for each $z\in\D$, we
have
    \begin{equation*}
    \begin{split}
    T_{g\circ \vp_a}(f)(\vp_a(z))
    &=\int_0^{\vp_a(z)}g'(\vp_a(\z))\vp'_a(\z)f(\z)\,d\z
    =\int_a^{z}g'(u)(f\circ\vp_a)(u)\,du\\
    &=-\int_0^{a}g'(u)(f\circ\vp_a)(u)\,du+\int_0^{z}g'(u)(f\circ\vp_a)(u)\,du\\
    &=-T_g(f\circ\vp_a)(a)+T_g(f\circ\vp_a)(z).
    \end{split}
    \end{equation*}
Therefore
    \begin{equation*}
    \begin{split}
    \|T_{g\circ \vp_a}(f)\|_{H^\infty}
    &=\sup_{z\in\D}\left| T_{g\circ \vp_a}(f)(z)\right|
    =\sup_{z\in\D}\left| T_{g\circ \vp_a}(f)(\vp_a(z))\right|\\
    &\le|T_g(f\circ\vp_a)(a)|+\sup_{z\in\D}\left| T_g(f\circ\vp_a)(z)\right|
    \le2\| T_g(f\circ\vp_a)\|_{H^\infty}\\
    &\le2\| T_g\|_{H^\infty\to H^\infty}\| f\|_{H^\infty},
    \end{split}
    \end{equation*}
and since $\|T_{gh}\|_{H^\infty\to H^\infty}=\|
T_{g}\|_{H^\infty\to H^\infty}$ for any constant function
$h\equiv\xi\in\T$, the assertion follows.
\end{proof}

The space $\BRV$ of analytic functions
with bounded radial variation consists of $g\in\H(\D)$ such that
    $$
    \sup_{0\le\theta<2\pi}V(g,\theta)=\sup_{0\le\theta<2\pi}\int_0^1|g'(te^{i\theta}|\,dt<\infty.
    $$
In \cite{AJS} the authors made an extensive study of the spaces
$T(H^\infty, H^\infty)$  and $T_c(H^\infty, H^\infty)$, in
particular they observed that $\BRV\subset T(H^\infty, H^\infty)$.
We do not know if this inclusion is strict, but as for
this question, we offer the following result.

\begin{theorem} Let $g\in\B$. Assume that there exists $l>0$ and a dense set $E\subset\T$ such that, for each $\xi\in E$, there exists a Jordan arc $\Gamma_\zeta$ in $\D\cup\{\xi\}$ from $0$ to $\xi$ with
    $$
    \int_{\Gamma_\xi}|g'(s)|\,|ds|\le l.
    $$
Then $T_g:H^\infty\to H^\infty$ is bounded.
\end{theorem}

\begin{proof} Let $f\in H^\infty$ be fixed and $h=T_g(f)$. Then $h\in\B$ with $\|h\|_\B\le\|f\|_{H^\infty}\|g\|_{\B}$.
Fix $\xi\in E$, and let $L=L(\xi)$ denote the lens, with vertexes at
$\xi$ and 0, formed by two circular arcs such that they meet $\T$ at
the angle $\pi/4$. If $z\in \Gamma_\xi$, then
    $$
    |h(z)|\le\int_{\Gamma_\zeta} |f(s)| |g'(s)|\, |ds|\leq \Vert
    f\Vert_\infty l
    $$
by the hypothesis. The curve $\Gamma_\xi$ intersects $\partial L$ at
$0$, $\xi$, and, perhaps, in another points. In any case we can
build domains $G_n$ such that $L\subset(\Gamma_\xi\cup_nG_n)$ and
$\partial G_n\subset(\Gamma_\xi\cup\partial L)$. Then
\cite[Theorem~4.2(iii)]{Po92}, see also~\cite{ACP}, gives
    $
    \textrm{dist}\,(h(z),h(\Gamma_\xi))\le\frac{e\pi}{4}\|f\|_{H^\infty}\|g\|_{\B}
    $
for all $z\in L$, and hence
    $$
    |h(z)|\le\sup_{t\in \Gamma_\xi}|h(t)|+\textrm{dist}\,
    (h(z),h(\Gamma_\xi))\le\left(l+\frac{e\pi}{4}\|g\|_{\B}\right)\|f\|_{H^\infty},\quad z\in L.
    $$
Now that $E$ is dense in $\T$, we have $\D\setminus\{
0\}=\cup_{\xi\in E}L(\xi)$, and thus $T_g:H^\infty\to H^\infty$ is
bounded.
\end{proof}

Regarding sufficient conditions,  it is known that $T_g:H^\infty\to
H^\infty$ is compact if $g\in \mathrm{BRV}_0$~\cite[Proposition
3.4]{AJS}. Recall that $\BRV_0$ is the Banach space of $g\in \H(\D)$
such that $g'$ is uniformly integrable on radii, that is, for each
$\varepsilon>0$, there exists $r=r(\e)\in(0,1)$ for which
    $$
    \int_r^1|g'(te^{i\theta})|dt\le\varepsilon, \quad 0\le\theta<2\pi.
    $$
It is clear that
    \begin{equation}\label{BRVT}
    \BRV_0\subset\BRV\subset T(H^\infty, H^\infty)\subset H^\infty.
    \end{equation}
The following result shows that the lacunary series with Hadamard
gaps are the same in all above-mentioned spaces of analytic
functions.

\begin{letterproposition}\label{pro:lacunary}
Let $g\in\mathcal{H}(\D)$ be a lacunary series with Hadamard gaps,
that is, $g(z)=a_0+\sum_{k=1}^\infty a_k z^{n_k}$ with
$\frac{n_{k+1}}{n_k}\ge\lambda>1$ for all $k\in\N$. Then the
following statements are equivalent:
\begin{enumerate}
\item[\rm(a)]$g\in\BRV_0$;
\item[\rm(b)]$g\in\BRV$;
\item[\rm(c)]$g\in H^\infty$;
\item[\rm(d)]$\sum_k|a_k|<\infty$;
\item[\rm(e)]$T_g:H^\infty \to H^\infty$ is compact;
\item[\rm(f)]$T_g:H^\infty \to H^\infty$ is bounded.
\end{enumerate}
\end{letterproposition}

\begin{proof}
By \cite[Vol I p.~247]{Zygmund59}, (c) implies (d), and (d) implies
(a) by \cite[Proposition~3.4]{AJS}. Hence the first four statements
are equivalent by the chain of inclusions \eqref{BRVT}. If
$g\in\BRV_0$, then $T_g:H^\infty \to H^\infty$ is compact by
\cite[Proposition~3.4]{AJS}. It follows that (e) and (f) are
equivalent to the first four statements by \eqref{BRVT}.
\end{proof}

Next we study the relationship between integral and Hankel
operators.
Given $\phi\in L^\infty(\T$) and a polynomial $f$, consider the
Hankel operator $H_\phi(f)=(I-P)(\phi f)$, where~$P$ is the Szeg\"o projection. If $1<p<\infty$ and
$g\in\H(\D)$, $H_{\overline{g}}$ and $T_g$ are simultaneously
bounded from $H^p$ to $L^p(\T)$, and this happens if and only if
$g\in\BMOA$ \cite{AC,Pell}. This is no longer true for  the extreme
points $p=1$ and $p=\infty$. For $p=1$ this fact follows
by~\cite{AC,JPSDuke84}. 
The next result deals with the case
$p=\infty$. Denote by $H(H^\infty,H^\infty)$ the set of symbols $g\in\H(\D)$ such that $H_{\overline{g}}:H^\infty\to H^\infty$ is bounded.

\begin{proposition}\label{singular}
The atomic singular inner function
$S(z)=\exp\left(\frac{z+1}{z-1}\right)$ belongs to $\BRV\setminus
H(H^\infty,H^\infty)$.
\end{proposition}

\begin{proof}
By
\cite[Theorem $6.6.3$ or Theorem~6.6.11]{CiMaRo}, the singular inner
function $S$ does not belong to the space of multipliers of
$\mathcal K$. Thus, by \cite[Proposition~6.1.5]{CiMaRo},
the Hankel operator with symbol $\overline S$, is not bounded on $H^\infty$. 
We will show that $S\in\BRV$, that is,
    \begin{equation}\label{hg1}
    \sup_{\theta\in (-\pi,\pi]}\left|V(S,\theta) \right|=\sup_{\theta\in
    (-\pi,\pi]}\int_0^1
    \frac{2}{|1-te^{i\theta}|^2}\exp\left(-\frac{1-t^2}{|1-te^{i\theta}|^2}\right)\,dt<\infty,
    \end{equation}
 and so $T_S:H^\infty\to H^\infty$ is bounded.

Since $1/4\le|1-te^{i\theta}|^2\le9/4$ for $0\le t\le 1/2$,
    \begin{equation*}
    \sup_{\theta\in (-\pi,\pi]}\int_0^{1/2}
    \frac{1}{|1-te^{i\theta}|^2}\exp\left(-\frac{1-t^2}{|1-te^{i\theta}|^2}\right)\,dt\le2.
    \end{equation*}
If $1/2<t<1$, then $\frac{ (1-t)^2+\theta^2}{\pi^2} \le
|1-te^{i\theta}|^2\le (1-t)^2+\theta^2$, and hence
    \begin{equation*}
    \int_{1/2}^1
    \frac{1}{|1-te^{i\theta}|^2}\exp\left(-\frac{1-t^2}{|1-te^{i\theta}|^2}\right)\,dt\le
    \pi ^2\int_{1/2}^1
    \frac{1}{(1-t)^2+\theta^2}\exp\left(-\frac{1-t}{(1-t)^2+\theta^2}\right)\,dt.
    \end{equation*}
If $|\theta|\ge1/2$, then
    \begin{equation*}
    \int_{1/2}^1
    \frac{1}{|1-te^{i\theta}|^2}\exp\left(-\frac{1-t^2}{|1-te^{i\theta}|^2}\right)\,dt\le2\pi^2,
    \end{equation*}
meanwhile for $|\theta|<1/2$, we have
    \begin{equation*}
    \begin{split}
    &\int_{1/2}^1
    \frac{1}{(1-t)^2+\theta^2}\exp\left(-\frac{1-t}{(1-t)^2+\theta^2}\right)\,dt
    = \int_{0}^{1/2}
    \frac{1}{x^2+\theta^2}\exp\left(-\frac{x}{x^2+\theta^2}\right)\,dx
    \\ & =
    \int_{0}^{|\theta|}
    \frac{1}{x^2+\theta^2}\exp\left(-\frac{x}{x^2+\theta^2}\right)\,dx+
    \int_{|\theta|}^{1/2}
    \frac{1}{x^2+\theta^2}\exp\left(-\frac{x}{x^2+\theta^2}\right)\,dx
    \\ & \le
    \int_{0}^{|\theta|}
    \frac{1}{\theta^2}\exp\left(-\frac{x}{2\theta^2}\right)\,dx+
    \int_{|\theta|}^{1/2}
    \frac{1}{x^2}\exp\left(-\frac{1}{2x}\right)\,dx
    \\ & =
    2\left(1-\exp\left(-\frac{1}{2|\theta|}\right)\right)+ 2\left(
    \exp\left(-1\right) -\exp\left(-\frac{1}{2|\theta|}\right)\right)
    \le 2(1+e^{-1}).
    \end{split}
    \end{equation*}
By combining these three estimates we deduce \eqref{hg1}.
\end{proof}

\section{Weakly compact integral operators mapping into $H^\infty$}\label{section:weakcompactness}

In this section,
 we deal with the weak compactness of integral operators mapping into~$H^\infty$.
  First of all, we characterize the weak-star topology in Banach spaces of analytic functions with a separable predual.

\subsection{Preliminary results on weak compactness}
Throughout this section a collection of results, which will be used
repeatedly in the paper, will be proved or recalled.

\begin{lemma}\label{le:wstartopology}
Let $X$ be a separable Banach space such that
$X^*\simeq Y\subset \H(\D)$ via a pairing $\langle \cdot, \cdot
\rangle$. Assume that
\begin{enumerate}
\item Each bounded sequence $\{f_n\}$ in $Y$ is uniformly bounded on compact subsets of~$\D$;
\item For every $z\in\D$, there exists $K_z\in X$ such that $f(z)=\langle K_z,f\rangle$ for all $f\in Y$.
\end{enumerate}
Then the weak-star topology $\sigma(Y,X)$ on $B_Y$ is
the topology of uniform convergence on compact subsets of $\D$.
\end{lemma}

\begin{proof}
Let $\{f_n\}$ be a bounded sequence in $Y$ such that it converges in
the weak-star topology $\sigma(Y,X)$ to $f\in Y$. Since $X^*\simeq Y$, $\{f_n\}$
converges pointwise to $f$ by (2). By Vitali's Theorem and (1),
$\{f_n\}$ converges uniformly on compact subsets of $\D$ to $f$.

Since $X$ is separable, the
weak-star topology $\sigma(Y,X)$ on $B_Y$ is metrizable by~\cite[p.~32]{Woj}.
Assume that $\{f_n\}$ is a sequence in $B_{Y}$ that converges  to
$f$ uniformly on compact subsets of $\D$ but $\{f_n\}$ does not
converge to $f$ in the weak-star topology $\sigma(Y,X)$. Let~$d$
denote a metric which induces the weak-star topology on $B_{Y}$. For
each $\ep>0$, there exists a subsequence $\{f_{n_k}\}$ such that
$d(f_{n_k}, f)\ge \ep$ for all $k$. Since $\{f_{n_k}\}\in B_{Y}$,
the Banach-Alaogl\'u Theorem shows that there exists a subsequence
$\{f_{{n_{k}}_{j}}\}$ that converges to some $g\in B_Y$ in the
weak-star topology. But, by arguing as in the first part of the
proof, we deduce that $\{f_{{n_{k}}_{j}}\}$ converges uniformly on
compact subsets of $\D$ to $g$. Thus $f=g$, and therefore
$\{f_{{n_{k}}_{j}}\}$ converges to $f$ in the weak-star topology,
that is, $d(f_{n_{k_j}}, f)\to0$, as $j\to\infty$. This
contradiction finishes the proof.
\end{proof}

Lemma~\ref{le:wstartopology} can be applied, for example, if $Y$ is any of the spaces $H^1$,
 $H^\infty$~\cite[Lemma~2.2]{Lefevre}, $\K$~\cite[Proposition~4.2.5]{CiMaRo}, $\B$, $\BMOA$, $A^1$, or $H_v^\infty$,
  which consists of $f\in\H(\D)$ such that $\sup_{z\in \D}|f(z)|v(z)<\infty$, where $v$ is a typical weight,
  that is, $v$ is continuous,  radial and $\lim_{r\to1^-}v(r)=0$~\cite{BBT}.   Moreover, it gives a description of the weak topology of the reflexives spaces $H^p$ and $A^p_\om$, $\om\in\R$ \cite{PelRatproj}.

Recall that $T:\mathcal{H}(\D)\to \mathcal{H}(\D)$ is continuous if $T(f_n)\to0$ uniformly on compact subsets of $\D$ whenever the same is true for $\{f_n\}$.

\begin{lemma}\label{biadjunto_general} Let $T:\mathcal{H}(\D)\to \mathcal{H}(\D)$ continuous and $X$ a Banach space such that $X^{**}$ is a space of analytic functions such that on its unit ball, the weak-star topology $\sigma (X^{**},X^*)$ coincides  with the topology of uniform convergence on compact subsets of $\D$.
If $T:X\to H^\infty$ is weakly compact, then $T^{**}(f)=T(f)$ for all
$f\in X^{**}$.
\end{lemma}

\begin{proof} Since $T:X\to H^\infty$ is weakly compact, the operator  $T^{**}:X^{**}\to H^\infty$ is bounded by \cite[p.~52]{Woj}.
Let $f\in X^{**}$. Goldstine's Theorem~\cite[p.~31]{Woj} and the
fact that the unit ball of $X^{**}$ endowed with the weak-star
topology is metrizable, imply that we can take a bounded sequence
$\{f_n\}$ in $X$ that converges to $f$ in the topology $\sigma
(X^{**},X^*)$. The weak compactness of $T$ implies that
$T(f_n)=T^{**}(f_n)$ converges to $T^{**}(f)\in H^\infty$ in the
weak topology, and then uniformly on compact subsets of $\D$. By the
hypothesis, $f_n\to f$ uniformly on compact subsets of $\D$. Thus
$T(f_n)\to T(f)$ uniformly on compact subsets of $\D$, and therefore
$T^{**}(f)=T(f)$.
\end{proof}

Since the weak compactness of an operator is equivalent to the weak compactness of its bi-adjoint by~\cite[p.~52]{Woj}, under the assumptions of Lemma~\ref{biadjunto_general}, $T_g:X\to H^\infty$ is weakly compact if and only if $T_g:X^{**}\to H^\infty$ is weakly compact. This can be applied, for example, if $X$ is $\B_0$, $\VMOA$ or $H_v^0$, which consists of $f\in\H(\D)$ such that $\lim_{|z|\to1^-}f(z)v(z)=0$, when $v$ is typical.
With regard to the last case, recall that $(H_v^0)^{**}\simeq H_v^\infty$ when $v$ is typical~\cite{BBT}.

The following well-known result can be proved by using
Eberlein-Smulian's theorem~\cite[p.~49]{Woj} and standard
techniques, so its proof is  omitted.

\begin{lemma}\label{le:weakdesc} Let $X,Y\subset\H(\D)$ be Banach spaces and let $T:X\to Y$ be a linear operator.
Assume that the following conditions are satisfied:
\begin{itemize}
\item[\rm(a)] The point evaluation functionals on $Y$ are bounded;
\item[\rm(b)] For every bounded sequence in $X$, there is a
subsequence which converges uniformly on
compact subsets of $\D$ to an element of $X$;
 \item[\rm(c)] If a sequence $\{f_n\}$ in $X$
converges uniformly on compact subsets of $\D$ to zero, then
$\{T(f_n)\}$ converges uniformly on compact subsets of $\D$ to zero.
\end{itemize}
Then $T:X\to Y$ is a compact (resp. weakly compact) if and only if $\{T(f_n)\}$ converges to zero in the norm
(resp. in the weak topology) of $Y$ for each bounded sequence
$\{f_n\}$ in $X$ such that $f_n\to 0$ uniformly on compact subsets
of~$\D$.
\end{lemma}

In~\cite[Proof of Theorem~3.5]{AJS} it is used an approach based on
the modulus of continuity of analytic functions~\cite{RST},
 in order  to show that the $g\in H^\infty$ such that
 $\lim_{\delta\to0^+}\|g-g^\delta\|_{H^\infty}=0$,
 belong to $A$, where $g^\delta(z)=g(e^{i\delta}z)$, for any $\delta\in \mathbb R$.
A proof of this result can  be obtained using ideas on semigroup
theory~\cite{BCDMPS}. Namely, if $g(z)=-iz$ then $T_g:H^\infty \to
A$ is bounded, so if one  considers the semigroup
$\{\varphi_t(z)=e^{it}z:t\ge0\}$, by \cite[Proposition 2]{BCDMPS} it
can be proved   that those $f\in H^\infty$ such that
$f\circ\varphi_t$ converges uniformly on $\D$ to $f$, belong to $A$.
The following immediate consequence is stated for further reference.

\begin{letterlemma}\label{g_in_A}
Let $g\in H^\infty$. Then $g\in A$ if and only if $\lim_{\delta\to0^+}\|g-g^\delta\|_{H^\infty}=0$.
\end{letterlemma}

Next some known results on operators defined on $H^\infty$ are
summarized. Recall that $T:H^\infty\to Y$ fixes a copy of the
sequence space~$\ell^\infty$, if there exists $Z\subset H^\infty$,
isomorphic to $\ell^\infty$, such that $T:Z\to Y$ is an isomorphism.

\begin{lettertheorem}\label{operators_on_Hinfty1}
Let $Y$ be a Banach space and $T:H^\infty\to Y$ bounded.
Then the following statements are equivalent:
\begin{enumerate}
\item[\rm(a)] $T$ is weakly compact;
\item[\rm(b)] $T$ is completely continuous, that is, $\Vert T(f_n)\Vert_Y \to 0$ for any bounded
sequence $\{f_n\}$ in $H^\infty$ that converges weakly to zero;
\item[\rm(c)] $T$ does not fix a copy of $\ell^\infty$.
\end{enumerate}
In particular, if $Y$ is separable, then $T$ is weakly compact.
\end{lettertheorem}

The fact that (a) implies (b) means that $H^\infty$ has the
Dunford-Pettis property, a result due to Bourgain \cite[Corollary~5.4]{Bourgain2}.
The converse implication can be seen in \cite[Theorem~3.1]{CDM_Contemporary}.
Being trivial that (a) implies (c), one can see that (c) implies (a)
in \cite[Theorem~7.1]{Bourgain}. We shall also use the
following result \cite[Theorem~1.2]{Lefevre}.

\begin{lettertheorem}\label{operators_on_Hinfty2}
Let $Y$ be a Banach space and $T:H^\infty\to Y^*$ a bounded
operator. Then $T:H^\infty\to Y^*$ is weakly compact and $w^*-w^*$-continuous if and only if $\Vert T(f_n)\Vert_{Y^*}\to 0$ for each bounded sequence $\{f_n\}$ in $H^\infty$ such that $\Vert f_n\Vert_{H^1}\to 0$.
\end{lettertheorem}

\subsection{Main results on weak compactness}

It is worth mentioning that $T_c(H^\infty,H^\infty)\subset A$ but
$T(H^\infty,H^\infty)\not\subset A$ by~\cite[Theorem~3.5 and
Proposition~2.12]{AJS}. Theorem~\ref{weakly_compact_imply_g_in_A}
improves the first of these results.

\medskip

\begin{Prf}{\em{Theorem~\ref{weakly_compact_imply_g_in_A}}}. For $f\in \mathcal{H}(\D)$ and  $\delta \in \mathbb{Re}$
define  $f^\delta(z)=f(e^{i\delta}z)$. Let $\{\delta_n\}$ be a
sequence of real numbers that converges to zero, and define
$h_n=g-g^{\delta_n}$ and $l_n=g-g^{-\delta_n}$. Then $\{h_n\}$ and
$\{l_n\}$ are bounded sequences in $H^\infty$ that converge to zero
uniformly on compact subsets of $\D$. Since $T_g:H^\infty\to
H^\infty$ is weakly compact, by passing to subsequences if
necessary, we may assume that $\{T_g(h_n)\}$ and $\{T_g(l_n)\}$
converge to zero in the weak topology of $H^\infty$. Further, since
$T_g:H^\infty\to H^1$ is weakly compact by the hypothesis,
Theorem~\ref{operators_on_Hinfty1} yields $\|T_g(h_n)\|_{H^1}\to 0$
and $\|T_g(l_n)\|_{H^1}\to 0$. Write $\rho_n(z)=e^{i\delta_n}z$.
Then
    \begin{eqnarray*}
    T_{g^{\delta_n}}(h_n)(z)&=& T_{g^{\delta_n}}(g-g^{\delta_n})(z)=
    \int_0^z (g(\xi)-g(e^{i\delta_n}\xi))g'(e^{i\delta_n}\xi) e^{i\delta_n}\, d\xi\\
    &=&\int_0^{e^{i\delta_n}z} (g(e^{-i\delta_n}u)-g(u))g'(u)\, du=
    -T_{g}(l_n)(e^{i\delta_n}z)=- T_g(l_n)\circ\rho_n(z),
    \end{eqnarray*}
and hence $\|T_{g^{\delta_n}}(h_n)\|_{H^1}=\|T_g(l_n)\circ\rho_n\|_{H^1}=\|T_g(l_n)\|_{H^1}\to0$. Since
$T_{h_n}(h_n)=T_g(h_n)-T_{g^{\delta_n}}(h_n)$, we deduce
    $$
    \frac{1}{2}\Vert  h_n^2\Vert_{H^1}=\Vert T_{h_n}(h_n)\Vert_{H^1}\le\Vert
    T_g(h_n)\Vert_{H^1}+\Vert T_{g^{\delta_n}}(h_n)\Vert_{H^1},
    $$
and therefore $\|h_n^2\|_{H^1}\to 0$. The same reasoning shows that $\|l_n^2\|_{H^1}\to 0$.

By Lemma~\ref{le:wstartopology}, the weak-star convergence of
bounded sequences in $H^\infty$ is nothing else but the uniform
convergence on compact subsets of $\D$, so it is clear that $T_g:
H^\infty\to H^\infty$ is $w^*-w^*$ continuous. Therefore Theorem~\ref{operators_on_Hinfty2} implies
$\|T_g(h_n^2)\|_{H^\infty}\to0$ and $\|T_g(l_n^2)\|_{H^\infty}\to0$.
By arguing as above, we deduce
$T_{g^{\delta_n}}(h_n^2)=-T_g(l_n^2)\circ\rho_n$, and thus $\Vert
T_{g^{\delta_n}}(h_n^2)\Vert_{H^\infty}\to 0$. Since
$T_{h_n}=T_g-T_{g^{\delta_n}}$, we finally obtain
    $$
    \frac{1}{3}\Vert  h_n^3\Vert_{H^\infty}=\Vert
    T_{h_n}(h_n^2)\Vert_{H^\infty}\le\Vert T_g(h^2_n)\Vert_{H^\infty}+\Vert
    T_{g^{\delta_n}}(h_n^2)\Vert_{H^\infty}\to 0,
    $$
which shows that $\Vert g-g^\delta\Vert_{H^\infty}\to0$, as $\delta\to0$. Therefore $g\in A$ by Lemma~\ref{g_in_A}.
\end{Prf}

\begin{proposition}\label{rang}
Let $X\subset\H(\D)$ be an admissible Banach space and $g\in A$ such that $T_g:X\to
H^\infty$ is weakly compact. Then $T_g:X\to A$ is bounded.
\end{proposition}

\begin{proof}
Let $f\in X$. Since $f_r$ converges to $f$ uniformly on compact
subsets of $\D$ and $\sup_{0<r<1}\|f_r\|_X<\infty$ by (P2),
$T_g(f_r)$ converges to $T_g(f)$ in the weak topology of $H^\infty$
by Lemma~\ref{le:weakdesc}. Since $T_g(f_r)\in A$ for all $r$ and
$A$ is closed in $H^\infty$ in the weak topology, by Mazur's theorem
\cite[p.~28]{Woj}, we deduce that $T_g(f)\in A$.
\end{proof}

\begin{corollary}\label{co:Xweak}
Let $g\in\H(\D)$ and $X\subset\H(\D)$ a Banach space such that $H^\infty\subset X$ and $X$ is isomorphic to $\ell^\infty$ or $H^\infty$. Then $T_g:X\to H^\infty$ is weakly compact if and only if $T_g:X\to A$ is bounded.
\end{corollary}

\begin{proof}
If $T_g:X\to H^\infty$ is weakly compact, so is $T_g:H^\infty\to H^\infty$, and hence $g\in A$ by Theorem~\ref{weakly_compact_imply_g_in_A}. Therefore $T_g:X\to A$ is bounded by Proposition~\ref{rang}. Conversely, assume that $T_g:X\to A$ is bounded. Since $X$ is isomorphic to $\ell^\infty$ or $H^\infty$, each bounded operator from $X$ into a separable Banach space is weakly compact by \cite[p.~156]{DiestelUhl} (if $X$ is isomorphic to $\ell^\infty$)
and Theorem~\ref{operators_on_Hinfty1} (if $X$ is isomorphic to $H^\infty$). Hence $T_g:X\to A$ is weakly compact.
\end{proof}

Recall that the Bloch space $\B$ is isomorphic to $\ell^\infty$.
Moreover, it is well-known that $H_v^\infty $ is isomorphic to
$\ell^\infty$ or $H^\infty$, depending on $v$, provided $v$ is
typical~\cite{Lusky}. We emphasize the result for the Bloch space.

\begin{corollary}
Let $g\in\H(\D)$ such that $T_g:\B\to H^\infty$ is bounded. Then the following
statements are equivalent:
\begin{enumerate}
\item[\rm(a)] $T_g:\B\to H^\infty$ is weakly compact;
\item[\rm(b)] $T_g:\B\to A$ is bounded;
\item[\rm(c)] $T_g:\B_0\to A$ is weakly compact.
\item[\rm(d)] $T_g:\B_0\to H^\infty$ is weakly compact.
\end{enumerate}
\end{corollary}
\begin{proof}
By Corollary~\ref{co:Xweak}, (a) and (b) are equivalent. If (c)
holds, then
   $T_g^{**}$ sends $ \B_0^{**}\simeq\B$ to $A$, so (b) follows from Lemma \ref{biadjunto_general}.
   The implication (d)$\Rightarrow$(a) can be proved analogously.
Being trivial that (a) and (b) imply (c) and (d), this finishes
the proof.\end{proof}

Corollary~\ref{co:Xweak} can also be applied to $H^\infty$. To do this, a result similar to Lemma~\ref{biadjunto_general} for $A$ will be needed.

\begin{lemma}\label{biadjunto} Let $T:\mathcal{H}(\D)\to \mathcal{H}(\D)$ continuous such that
$T:A\to H^\infty$ is weakly compact. Then $T^{**}(f)=T(f)$ for all
$f\in H^\infty$.
\end{lemma}

\begin{proof}
Recall that $T:A\to H^\infty$ is weakly compact if and only if
$T^{**}:A^{**}\to H^\infty$ is bounded~\cite[Theorem~6(c), p.~52]{Woj}. Let $f\in
H^\infty$. Since $H^\infty\subset A^{**}$ by \eqref{eq:A^**},
Goldstine's Theorem~\cite[p. 31]{Woj} implies that there exists a
bounded net $\{f_\delta\}$ in $A$ such that $\{f_\delta\}$ converges
to $f$ in the weak-star topology $\sigma(A^{**},A^*)$. Since $T:A\to
H^\infty$ is weakly compact, we deduce that
$T(f_\delta)=T^{**}(f_\delta)$ converges to $T^{**}(f)\in A$ in the
weak topology $\sigma(H^\infty,(H^\infty)^*)$. It follows that
$T^{**}(f_\delta)\to T^{**}(f)$ in the weak-star topology of
$H^\infty$, and hence uniformly on compact subsets of $\D$ by
Lemma~\ref{le:wstartopology}. Since $f_\d\in A$ and
$\K_a\subset\K\simeq A^*$, the $\sigma(A^{**},A^*)$-convergence
implies the convergence of $\{f_\d\}$ to $f$ in the weak-star
topology $\sigma(H^\infty,\K_a)$. Therefore $f_\delta\to f$
uniformly on compact subsets of $\D$ by
Lemma~\ref{le:wstartopology}, and further, by the hypotheses,
$T(f_\delta)\to T(f)$ uniformly on compact subsets of $\D$. Hence
$T^{**}(f)=T(f)$.
\end{proof}

\begin{Prf}{\em{Corollary~\ref{co:weakAA}.}}
The statements (a) and (b) are equivalent by
Corollary~\ref{co:Xweak}. If (b) is satisfied, then $T_g:A\to A$ is
bounded and by (a), it is weakly compact.  Conversely, assume that
$T_g:A\to A$ is weakly compact. Then $T_g^{**}:A^{**}\to A$ is
bounded by~\cite[p.~52]{Woj}, and hence $T_g^{**}:H^\infty\to A$ is
bounded by \eqref{eq:A^**}. But $T_g^{**}=T_g$ on $H^\infty$ by
Lemma~\ref{biadjunto}, and so $T_g:H^\infty\to A$ is bounded. Thus
(b) and (c) are equivalent.
\end{Prf}

\begin{corollary}\label{le:XAinfty}
Let $X\subset\H(\D)$ be a Banach space such that $H^\infty\subset X$.
\begin{itemize}
\item[\rm(i)] If $X$ is separable, then $T(X,H^\infty)\subset A$.
\item[\rm(ii)] If $X$ is reflexive and admissible, then $T(X,H^\infty)=T(X,A)$.
\end{itemize}
\end{corollary}

\begin{proof} (i) Consider the identity operator $i(f)=f$. By the hypothesis, $i:H^\infty\to X$ is bounded, and hence weakly compact by Theorem~\ref{operators_on_Hinfty1}. Now that $T_g:H^\infty \to H^\infty$ is the composition of $i:H^\infty\to X$
and $T_g:X\to H^\infty$, $T_g:H^\infty\to H^\infty$ is weakly compact, and therefore $g\in A$ by Theorem~\ref{weakly_compact_imply_g_in_A}.

(ii) Clearly, $T(X,A)\subset T(X,H^\infty)$. If $g\in T(X,H^\infty)$ and $X$ is reflexive, then $i:H^\infty\to X$ is weakly compact,
and the proof of (i) shows that $g\in A$. Hence $g\in T(X,A)$ by Proposition~\ref{rang}.
\end{proof}

We observe that Theorem~\ref{Thm-main-X}
together with Corollary~\ref{le:XAinfty} implies
Theorem~\ref{Thm-main-Xintro}. It is also worth mentioning that
Corollary~\ref{le:XAinfty} can be applied to the Hardy spaces and
weighted Bergman spaces induced by radial weights. In particular,
Corollary~\ref{le:XAinfty}(ii) together with
\cite[Corollary~7]{PelRatproj} implies the equivalence between (a)
and (b) in Theorem~\ref{Thm:g-is-constantintro}. Moreover, when $v$ is a typical weight, Corollary~\ref{le:XAinfty}(i) shows that $T(H_0^v,H^\infty)\subset A$.

We finish this section by asking whether or not each bounded operator $T_g:
A\to A$ is weakly compact. We do not know the answer to this
question, but we show that $T_g:A\to A$ is weakly
compact if $g\in\BRV\cap A$.

\begin{theorem} \label{pr:BRVA}
Let $g\in\H(\D)$ such that $T_g:A\to A$ is bounded and
$V(g,\theta)<\infty$ for all $\theta$. Then $T_g:A\to A$ is a weakly
compact, that is, $T_g:H^\infty\to A$ is bounded. In particular, if
$g\in\BRV\cap A$, then $T_g:A\to A$ is weakly compact.
\end{theorem}

A preliminary result is needed for the proof of Theorem~\ref{pr:BRVA}.

\begin{lemma}\label{le:weakcompactnessA}
Let $X\subset\H(\D)$ be a Banach space such that for every bounded sequence in $X$, there exists a
subsequence which converges uniformly on
compact subsets of $\D$ to an element of $X$. Assume that $T_g:X\to A$ is bounded. Then $T_g$ is weakly compact if
and only $T_g(f_n)(\xi)\to0$ for all $\xi\in\T$ and all bounded sequences $\{f_n\}$ in $X$ that converge to
zero uniformly on compact subsets of $\D$.
\end{lemma}

\begin{proof}
Assume first that $T_g:X\to A$ is weakly compact, and let $\{f_n\}$ be a bounded sequence in $X$ that converges to zero uniformly on compact
subsets of $\D$. Then $\{T_g(f_n)\}\to0$ in the weak topology of $A$ by Lemma~\ref{le:weakdesc}. Since the point evaluation functional $\delta_\xi(f)=f(\xi)$ is bounded in $A$ for each $\xi\in\overline{\D}$, the weak convergence implies $T_g(f_n)(\xi)=\delta_\xi(T_g(f_n))\to0$ for all $\xi\in\T$.

To see the converse, let $\{f_n\}$ be a bounded sequence in
$X$ that converges to zero uniformly on compact subsets of $\D$. Let
$C(\overline \D)$ denote the Banach space of complex-valued
continuous functions on $\overline\D$. Since $A\subset
C(\overline\D)$, the Hahn-Banach Theorem shows that $T_g(f_n)$
converges to zero in the weak topology of $A$, if and only if,
$T_g(f_n)$ converges to zero in the weak topology of $C(\overline
\D)$. Since $\{T_g(f_n)\}$ is a bounded sequence in $C(\overline
\D)$,  then $T_g(f_n)$ converges to zero in the weak topology of
$C(\overline \D)$ if and only if $\lim_{n\to\infty}T_g(f_n)(z)=0$
for all $z\in \overline \D$, \cite[Theorem 1, p.~66]{Diestel}.
 This  last fact happens by the hypotheses.
\end{proof}

\begin{Prf}{\em{Theorem~\ref{pr:BRVA}}.} We begin with showing that if $V(g,\theta)<\infty$, $f\in H^\infty$ and $T_g(f)(e^{i\theta})$ is well-defined, then
    \begin{equation}\label{pillu}
    T_g(f)(e ^{i\theta})=e^{i\theta}\int_0^1 f(re^{i\theta})g'(re^{i\theta})\,dr.
    \end{equation}
Note first that under these hypotheses, $\lim_{r\to
1^-}T_g(f)(re^{i\theta})=T_g(f)(e^{i\theta})$. For $n\in\N$, define
$h_n(r)=e^{i\theta} f(re^{i\theta})g'(re^{i\theta})\chi
_{[0,1-1/n)}$, where $\chi _{[0,1-1/n)}$ denotes the characteristic
function of $[0,1-1/n)$. Clearly, $h_n(r)$ converges to $e^{i\theta}
f(re^{i\theta})g'(re^{i\theta})$ for all $0<r<1$, and
    $$
    |h_n(r)|\le\|f\|_{H^\infty}|g'(re^{i\theta})|, \quad 0<r<1.
    $$
Since $|g'(re^{i\theta})|$ is integrable, the Dominated Convergence
Theorem gives
    \begin{equation*}
    \begin{split}
    T_g(f)((1-1/n)e^{i\theta})
    &=e^{i\theta}\int_0^{1-1/n}
    f(re^{i\theta})g'(re^{i\theta})dr\\
    &=\int_0^1h_n(r)dr\to
    e^{i\theta}\int_0^1 f(re^{i\theta})g'(re^{i\theta})dr, \quad  n\to\infty.
    \end{split}
    \end{equation*}

To prove the assertion, let $\{f_n\}\in A$ such that
$\sup_n\|f_n\|_A\le1$ and $\{f_n\}$ converges to zero uniformly on
compact subsets of $\D$. Let $\xi\in\T$ be fixed. Since
    $
    |f_n(r\xi)g'(r\xi)|\le|g'(r\xi)|
    $
and, for each $0<r<1$, $|f_n(r\xi)g'(r\xi)|\to0$, as $n\to\infty$, the Dominated Convergence Theorem and \eqref{pillu} yield
    $$
    |T_g(f_n)(\xi)|\le\int_0^1 |f_n(r\xi)g'(r\xi)|\,dr\to0.
    $$
Hence $T_g:A\to A$ is weakly compact by Lemma~\ref{le:weakcompactnessA}.
\end{Prf}

\section{Compact integral operators mapping into $H^\infty$}\label{section:compactness}

We begin this section with a characterization of the compactness of
$T_g$ that is of the same spirit as Theorem~\ref{Thm-main-X}. To do
this, we will need the following lemma concerning compact operators
mapping into $A$.

\begin{lemma}\label{le:compactA}
Let $X,Y\subset\H(\D)$ such that $Y\simeq X^*$ via the
$H(\beta)$-pairing. Assume that for every bounded sequence in $X$,
there is a subsequence which converges uniformly on compact subsets
of $\D$ to an element of $X$. Let $T:X\to A$ be bounded such that
$T(f_n)\to0$ uniformly on compact subsets of $\D$ for each bounded
sequence $\{f_n\}$ in $X$ such that $f_n\to0$ uniformly on compact
subsets of $\D$. Then $T:X\to A$ is compact if and only if
$\{T^*(K^{H(\beta)}_z):z\in\D\}$ is relatively compact in $Y$.
\end{lemma}

\begin{proof}
Assume first that $T:X\to A$ is compact. Then $T^*:A^*\to X^*$ is
compact. Consider the space
$$
Z=\left\{f\in \H(\D): \sum_{n=0}^\infty \beta_n\hat f(n) z^n\in \K\right\}.
$$ Using $\lim_{n\to\infty}\sqrt[n]{\beta_n}= 1$, one can proof that  $A^*\simeq Z$ via the
$H(\beta)$-pairing. Since $\|K^{H(\beta)}_z\|_{Z}=1$ for all
$z\in\D$, the set $\{T^*(K^{H(\beta)}_z):z\in\D\}$ is relatively
compact in $Y$.

Assume now that $\{T^*(K^{H(\beta)}_z):z\in\D\}$ is relatively
compact in $Y$, and suppose on the contrary to the assertion that
$T:X\to A$ is not compact. Then, by Lemma~\ref{le:weakdesc}, there
exist $\varepsilon>0$ and a sequence $\{f_n\}$ in $X$ such that
$\|f_n\|_X=1$ and $f_n\to0$ uniformly on compact subsets of $\D$,
but $\|T(f_n)\|_A\ge2\varepsilon$ for all $n$. It follows that for
each $n$, there exists $z_n\in\D$ such that
$|T(f_n)(z_n)|\ge\varepsilon$. Since
$\{T^*(K^{H(\beta)}_z):z\in\D\}$ is relatively compact by the
hypothesis, there exists a subsequence $\{z_{n_k}\}$ and $h\in Y$
such that $T^*(K^{H(\beta)}_{z_{n_k}})$ converges to $h$ in the norm
topology of $Y$. Hence, there exists $N\in\N$ such that
    $
    |T(f_m)(z_{n_k})-\langle f_m,h\rangle_{H(\beta)}|<\varepsilon/4
    $
for all $m\in\N$ and $k\ge N$. Since $f_n\to0$ uniformly on compact
subsets of $\D$, $\lim_{m\to \infty}T(f_m)(z_{n_{N}})=0$ by the
hypothesis, and hence $\limsup_{m\to\infty}|\langle
f_m,h\rangle_{H(\beta)}|\le\varepsilon/4$. On the other hand,
    \begin{equation*}
    \begin{split}
    |\langle f_{n_k},h\rangle_{H(\beta)}|&=|T(f_{n_k})(z_{n_k}) -(T(f_{n_{k}})(z_{n_k})-\langle f_{n_{k}},h\rangle_{H(\beta)})|\\
&\geq |T(f_{n_k})(z_{n_k})|-|T(f_{n_{k}})(z_{n_k})-\langle
f_{n_{k}},h\rangle_{H(\beta)}|
> 3\varepsilon /4, \quad
 k\ge N.
    \end{split}
    \end{equation*}
This leads to a contradiction, and therefore $T:X\to A$ is compact.
\end{proof}

\begin{Prf}{\em{Theorem~\ref{th:compactH2pairing}}.}
In order to prove the first assertion, assume that $T_g:X\to
H^\infty$ is compact. Then $g\in A$ by
Theorem~\ref{weakly_compact_imply_g_in_A}. Hence $T_g:X\to A$ is
bounded by Proposition~\ref{rang}, and thus $T_g:X\to A$ is compact.
 The second assertion follows by Lemma~\ref{le:compactA}.
\end{Prf}

\medskip

It is known that $T_c(H^\infty,H^\infty)\subset
T_c(A,A)$~\cite[Proposition~3.7]{AJS}. The following result shows,
in particular, that $T_c(H^\infty,H^\infty)=T_c(A,A)$, and hence
gives an affirmative answer to~\cite[Problem~4.4]{AJS}.

\medskip

\begin{Prf}{\em{Theorem~\ref{compactness_Hinfty}.}}
An application of Theorem~\ref{th:compactH2pairing}, with
$X=H^\infty$, shows that (i) is equivalent to (iii). Since clearly
(iii) implies (ii) which in turn implies (iv), in order to show the
equivalences of the  first four statements, it is enough to show that (iv)
implies (i). To see this, note first that if $T_g:A\to H^\infty$ is
compact,
then $T_g^{**}: A^{**}\to (H^\infty)^{**}$ is compact. Since
$A^{**}=H^\infty\bigoplus \K_s^*$ and $T_g^{**}$ coincides with
$T_g$ in $H^\infty$ by Lemma~\ref{biadjunto}, $T_g:H^\infty\to
(H^\infty)^{**}$ is compact. Now that $T_g: H^\infty\to H^\infty$ is
bounded by the hypothesis and Theorem~\ref{Thm:main-1}(ii), we
deduce that $T_g: H^\infty\to H^\infty$ is compact.

Let us see that (i) implies (v). Since  $T_g:H^\infty\to H^\infty$
is $w^*-w^*$ continuous, there is an operator $S:\K_a\to \K_a$ such
that $S^*=T_g$. Since $T_g $ is compact, we also have that~$S$ is a
compact operator. Using that $\|K^{H^2}_z\|_{\K_a}=1$ for all
$z\in\D$, the set
$\{T_g^*(K^{H^2}_z):z\in\D\}=\{S(K^{H^2}_z):z\in\D\}$ is relatively
compact in $\K$. Conversely, assume that (v) is satisfied. By
Theorem \ref{Thm:main-1}(ii), the operator $T_g:H^\infty\to
H^\infty$ is bounded. Suppose that  is not compact. Then, by
Lemma~\ref{le:weakdesc}, there exist $\varepsilon>0$ and a bounded
sequence $\{f_n\}$ in $H^\infty$ such that $\|f_n\|_{H^\infty}=1$
and $f_n\to0$ uniformly on compact subsets of $\D$, but
$\|T_g(f_n)\|_{H^\infty}\ge2\varepsilon$ for all $n$. It follows
that for each $n$, there exists $z_n\in\D$ such that
$|T_g(f_n)(z_n)|\ge\varepsilon$. Since
$\{T_g^*(K^{H^2}_z):z\in\D\}$ is relatively compact in $\K$,
and then in $\K_a$,  there exists a subsequence $\{z_{n_k}\}$ and
$h\in \K_a$ such that $T_g^*(K^{H^2}_{z_{n_k}})$ converges to~$h$ in the norm of  $\K$. We get a contradiction arguing as in the
proof of Lemma \ref{le:compactA}. So $T_g:H^\infty\to H^\infty$ is
compact.\end{Prf}

\medskip

Next, we apply Theorem~\ref{th:compactH2pairing}  to the cases
$X=H^p$, $1\le p<\infty$ and $X=\BMOA$. We denote $D(a,r)=\{z\in\C:\,
|z-a|<r\}$, $a\in\C$,\, $r>0$.
\begin{theorem}\label{th:BMOAcompact}
Let $g\in\H(\D)$.
\begin{itemize}
\item[\rm(i)] If $1<p<\infty$, then the following statements are equivalent:
\begin{itemize}
\item[\rm(a)] $T_g:H^p\to H^\infty$ is compact;
\item[\rm(b)] $T_g:H^p\to A$ is compact;
\item[\rm(c)] $\{G^{H^2}_{g,z}:z\in\D\}$ is relatively compact in $H^{p'}$;
\item[\rm(d)]
    $\displaystyle
    \lim_{R\to1^-}\sup_{z\in\D}\int_{\T}\left(\int_{\Gamma(\xi)\setminus
    D(0,R)}\left|(G^{H^2}_{g,z})'(w)\right|^2\,dA(w)\right)^{\frac{p'}{2}}|d\xi|=0,
    $
    where
    $$\Gamma(\xi)=\left\{\z\in\D:|\t-\arg\z|<\frac12\left(1-|\z|\right)\right\}$$ is the lens-type region with vertex at $\xi=e^{i\t}$.
\end{itemize}
\item[\rm(ii)] If $0<p\le1$, then $T_g:H^p\to H^\infty$ is compact if and only if $g$ is constant.
\item[\rm(iii)]
The following statements are equivalent:
\begin{enumerate}
\item[(a)] $T_g:\BMOA\to H^\infty$ is compact;
\item[(b)] $T_g:\VMOA\to A$ is compact;
\item[(c)] $T_g:\BMOA\to A$ is compact;
\item[(d)] $T_g:\VMOA\to H^\infty$ is compact;
\item[(e)] $\{ G^{H^2}_{g,z}:z\in \D\}$ is a relatively compact in $H^1$;
\item[(f)] $
    \displaystyle\lim_{R\to1^-}\sup_{z\in\D}\int_{\T}\left(\int_{\Gamma(\xi)\setminus D(0,R)}\left|(G^{H^2}_{g,z})'(w)\right|^2\,dA(w)\right)^{\frac{1}{2}}|d\xi|=
    0.
    $
\end{enumerate}
\end{itemize}
\end{theorem}

With the aim of proving Theorem~\ref{th:BMOAcompact},  some notation
and known results are introduced.  For a positive Borel measure
$\nu$ on $\D$, finite on compact sets, denote the usually called
square function by
    $$
    A_{2,\nu}(f)(\z)= \left(\int_{\Gamma(\z)}|f(z)|^2\,d\nu(z)\right)^{\frac12}.
    $$
For $0<p<\infty$, the tent space $T^p_2(\nu)$ consists of the
$\nu$-equivalence classes of $\nu$-measurable functions $f$ such
that $\|f\|_{T^p_2(\nu)}=\|A_{2,\nu}(f)\|_{L^p(\T)}$ is finite. It
is well known that neither the opening nor the exact shape of
$\Gamma(\xi)$ is relevant, so $\Gamma(\xi)$ can be replaced by any
``reasonable Stolz-type domain'' in the definition of $T^p_2(\nu)$
without changing the space~\cite{CMS} (see
also~\cite[Lemma~C]{Pel}). Define
    $$
    C^2_{2,\nu}(f)(\z)=\sup_{a\in\Gamma(\z)}\frac{1}{|I(a)|}\int_{T(a)}|f(z)|^2 (1-|z|)\,d\nu(z),\quad \z\in\T,
    $$
where $T(a)$ is a tent induced by the point $a\in\D\setminus\{0\}$.
A quasi-norm in the tent space $T^\infty_2(\nu)$ is defined by
$\|f\|_{T^\infty_2(\nu)}=\|C_{2,\nu}(f)\|_{L^\infty(\T)}$. Let
$dh(z)=dA(z)/(1-|z|^2)^2$ denote the hyperbolic measure, and
recall the well-known relations
    \begin{equation}\label{Eq:H^p-norm}
    \|f\|^{p}_{H^{p}}\asymp
\|f'(z)(1-|z|)\|^p_{T^p_2(h)}+|f(0)|^{p},\quad f\in\H(\D),
    \end{equation}
for $0<p<\infty$, and
    \begin{equation}\label{eq:BMOAnormequiv}
    \|f\|^{2}_{\BMOA}\asymp\|f'(z)(1-|z|^2)\|^2_{T^\infty_2(h)}+|f(0)|^2,\quad f\in\H(\D).
    \end{equation}
The original source for the theory of tent spaces is \cite{CMS}.

The proof of Theorem~\ref{th:BMOAcompact}(iii) is based on the
following stopping-time argument
 employed in the proof of
\cite[Theorem~1]{CMS}. For $\z\in\T$ and $0\le h\le\infty$, let
    \begin{equation*}
    \begin{split}
    \Gamma^h(\z)&=\Gamma(\z)\setminus \overline{D\left(0,\frac{1}{1+h}\right)}
    =\left\{z\in\D:|\arg z-\arg \z|<\frac{1-\left|z\right|}{2}<\frac{h}{2(1+h)}\right\}
    \end{split}
    \end{equation*}
and
    $$
    A^{2}_{2,\nu}(g|h)(\z)=\int_{\Gamma^h(\z)}|g(z)|^{2}\,d\nu(z),\quad \z\in\T.
    $$
For every $g\in T^\infty_{2}(\nu)$ and $\z\in\T$, define the
stopping time by
    $$
    h(\z)=\sup\left\{h:A_{2,\nu}(g|h)(\z)\le C_1C_{2,\nu}(g)(\z)\right\},
    $$
where $C_1>0$ is an appropriately chosen (large) constant. Then, by \cite[(4.3)]{CMS} (see
\cite[(2.3)]{Pel} for details in the case of $\D$), there
exists a constant $C_2>0$ such that
    \begin{equation}\label{Eq:stopping-time}
    \int_\D k(z)(1-|z|)\,d\nu(z)\le
    C_2\int_\T\left(\int_{\Gamma^{h(\z)}(\z)}k(z)\,d\nu(z)\right)\,|d\z|,
    \end{equation}
for all $\nu$-measurable non-negative functions $k$. With these preparations we are ready for the proof.

\medskip

\begin{Prf}{\em{Theorem~\ref{th:BMOAcompact}}.}
(i) The first three statements are equivalent by
Theorem~\ref{th:compactH2pairing}. To deal with (d), note first that, by
Theorem~\ref{Thm:main-1} and \eqref{Eq:H^p-norm},
$T_g:H^p\to H^\infty$ is bounded if and only if
    $$
    \sup_{z\in\D}\int_{\T}\left(\int_{\Gamma(\xi)}\left|(G^{H^2}_{g,z})'(\z)\right|^2\,dA(\z)\right)^{\frac{p'}{2}}|d\xi|<\infty.
    $$
If $f\in H^p$, then Green's theorem gives
    \begin{equation}\label{muna}
    \langle f,G^{H^2}_{g,z}\rangle_{H^2}=2\int_\D f'(\z)\overline{(G^{H^2}_{g,z})'(\z)}\log\frac1{|\z|}\,dA(\z)+f(0)\overline{G^{H^2}_{g,z}(0)}.
    \end{equation}
For $\z\in\D\setminus D(0,e^{-1})$, let
$I(\z)=\{e^{i\t}:|\t-\arg\z|<\frac12\log\frac1{|\z|}\}$. If $e^{-1}<R<1$, then
Fubini's theorem shows that
    \begin{equation}\label{muna2}
    \begin{split}
    &\int_{\D\setminus D(0,R)} f'(\z)\overline{(G^{H^2}_{g,z})'(\z)}\log\frac1{|\z|}\,dA(\z)\\
    &=\int_{\D\setminus D(0,R)} f'(\z)\overline{(G^{H^2}_{g,z})'(\z)}\int_{I(\z)}|d\xi|\,dA(\z)\\
    &=\int_\T\left(\int_{\Gamma'(\xi)\setminus D(0,R)}f'(\z)\overline{(G^{H^2}_{g,z})'(\z)}\,dA(\z)\right)|d\xi|,
    \end{split}
    \end{equation}
where $\Gamma'(\xi)=\{\z\in\D:|\arg\xi-\arg\z|<\frac12\log\frac1{|\z|}\}$ are lens-type regions.

If (c) is satisfied, then for given $\e>0$, there exist
$N=N(\e)\in\N$ and $h_1,\dots,h_N\in H^{p'}$ such that for each
$z\in\D$, we have $\|G^{H^2}_{g,z}-h_j\|_{H^{p'}}<\e$ for
some $j=j(z)\in\{1,\ldots,N\}$. It follows that
    \begin{equation*}
    \begin{split}
    &\int_{\T}\left(\int_{\Gamma(\xi)\setminus D(0,R)}\left|(G^{H^2}_{g,z})'(w)\right|^2\,dA(w)\right)^{\frac{p'}{2}}|d\xi|\\
    &\lesssim\e^{p'}+\int_{\T}\left(\int_{\Gamma(\xi)\setminus D(0,R)}\left|h_j'(w)\right|^2\,dA(w)\right)^{\frac{p'}{2}}|d\xi|,
    \end{split}
    \end{equation*}
and by choosing $R$ sufficiently close to 1, we deduce (d).

Assume now (d), and let $\{f_n\}$ be a uniformly norm bounded family
in $H^p$ such that $f_n\to0$ uniformly on compact subsets of $\D$.
To see that $T_g:H^p\to H^\infty$ is compact, by
Lemma~\ref{le:weakdesc}, it suffices to show that
$\|T_g(f_n)\|_{H^\infty}\to0$. Let $\e>0$,
$C_1=\sup_{z\in\D}\|G^{H^2}_{g,z}\|_{H^{p'}}^{p'}<\infty$,
$C_2=\sup_{n}\|f_n\|_{H^{p}}<\infty$ and $R=R(\e)\in(e^{-1},1)$ such
that
    $$
    \sup_{z\in\D}\int_{\T}\left(\int_{\Gamma(\xi)\setminus D(0,R)}\left|(G^{H^2}_{g,z})'(\z)\right|^2\,dA(\z)\right)^{\frac{p'}{2}}|d\xi|<\e^{p'},
    $$
and $N=N(R,\e)\in\N$ such that $|f_n(0)|,|f'_n(\z)|<\e$ for all
$n\ge N$ and $\z\in \overline{D(0,R)}$. Then, for $n\ge N$,
\eqref{le:tgformula}, \eqref{muna}, \eqref{muna2} and H\"older's
inequality give
    \begin{equation*}
    \begin{split}
    \|T_g(f_n)\|_{H^\infty}
    &\lesssim\sup_{z\in\D}\left|\int_\D f_n'(\z)\overline{(G^{H^2}_{g,z})'(\z)}\log\frac1{|\z|}\,dA(\z)\right|+|f_n(0)|\|g\|_{H^\infty}\\
    &\lesssim\sup_{z\in\D}\int_\T\left(\int_{\Gamma'(\xi)\setminus D(0,R)}|f_n'(\z)||(G^{H^2}_{g,z})'(\z)|\,dA(\z)\right)|d\xi|
    +C_1\e \\
        &\le\sup_{z\in\D}\left(\int_{\T}\left(\int_{\Gamma'(\xi)\setminus D(0,R)}\left|(G^{H^2}_{g,z})'(w)\right|^2\,dA(w)\right)^{\frac{p'}{2}}|d\xi|\right)^\frac1{p'}
    \\
    &\quad\cdot\left(\int_\T\left(\int_{\Gamma'(\xi)\setminus
    D(0,R)}|f_n'(\z)|^2\,dA(\z)\right)^\frac{p}{2}|d\xi|\right)^\frac1p+C_1\e\\
    &\lesssim (C_1+C_2)\e,
    \end{split}
    \end{equation*}
and it follows that $T_g:H^p\to H^\infty$ is compact. Thus (a) is
satisfied and the proof is complete.

(ii) Let $g\in\H(\D)$ be non-constant. It suffices to consider the
case $p=1$. If $T_g:H^1\to H^\infty$ is compact, then $T_g:H^1\to A$
is compact by Theorem~\ref{th:compactH2pairing}. Hence
$T_g^*:A^*\to(H^1)^*$ is compact and it follows that
$\{G_{z,g}^{H^2}:z\in\D\}$ is relatively compact in $\BMOA$. Hence,
for given $\e>0$, there exist $z_1,\ldots,z_N\in\D$ such that for
each $z\in\D$, we have
$\|G^{H^2}_{g,z}-G^{H^2}_{g,z_j}\|_{\BMOA}<\e$ for some
$j=j(z)\in\{1,\ldots,N\}$. Since
    \begin{equation*}
    \begin{split}
    &\sup_{a\in\D}\frac{1}{1-|a|}\int_{S(a)\setminus D(0,R)}|(G_{z,g}^{H^2})'(w)|^2(1-|w|^2)\,dA(w)\\
    &\lesssim\|G^{H^2}_{g,z}-G^{H^2}_{g,z_j}\|^2_{\BMOA}+\sup_{a\in\D}\frac{1}{1-|a|}\int_{S(a)\setminus D(0,R)}|(G^{H^2}_{g,z_j})'(w)|^2(1-|w|^2)\,dA(w)\\
    &\lesssim\e^2+\sup_{|a|\ge R}\frac{1}{1-|a|}\int_{S(a)}|(G^{H^2}_{g,z_j})'(w)|^2(1-|w|^2)\,dA(w),
    \end{split}
    \end{equation*}
  and  $G^{H^2}_{g,\z}\in A \subset\VMOA$ for each
$\z\in\D$,  we deduce that
    \begin{equation*}
    \begin{split}
    \lim_{R\to1^-}\sup_{a,z\in\D}\frac{1}{1-|a|}\int_{S(a)\setminus D(0,R)}|(G_{z,g}^{H^2})'(w)|^2(1-|w|^2)\,dA(w)=0,
    \end{split}
    \end{equation*}
which is equivalent to
    \begin{equation*}
    \begin{split}
    \lim_{R\to1^-}\sup_{a,z\in\D}\int_{\D\setminus D(0,R)}|(G_{z,g}^{H^2})'(w)|^2(1-|\vp_a(w)|^2)\,dA(w)=0.
    \end{split}
    \end{equation*}
By choosing $a=z$, we obtain
    \begin{equation}\label{Eq:condition-ugly}
    \lim_{R\to1^-}\sup_{z\in\D}\left((1-|z|^2)\int_{\D\setminus D(0,R)}\left|\frac{\int_0^z\frac{g'(\z)}{(1-\overline{w}\z)^2}\,d\z}{1-\overline{w}z}\right|^2(1-|w|^2)\,dA(w)\right)=0.
    \end{equation}
Now \eqref{eq:bi2} gives
\begin{equation*}
    \begin{split}
    \frac{\int_0^z\frac{g'(\z)}{(1-\overline{w}\z)^2}\,d\z}{1-\overline{w}z}
    &=\left(\sum_{k=0}^\infty\left(\sum_{n=0}^\infty(n+1)\widehat{g}(n+1) \frac{z^{n+k+1}}{n+k+1}\right)\overline{w}^k(k+1)\right)
    \left(\sum_{j=0}^\infty z^j\overline{w}^j\right)\\
    &=\sum_{m=0}^\infty\left(\sum_{k=0}^m\left(\sum_{n=0}^\infty(n+1)\widehat{g}(n+1) \frac{z^{n+1}}{n+k+1}\right)(k+1)z^{m}\right)\overline{w}^m,
    \end{split}
    \end{equation*}
and thus, by choosing $N\in\N$ such that $\widehat{g}(N+1)\ne0$, the
supremum in \eqref{Eq:condition-ugly} can be estimated downwards as
follows
    \begin{equation*}
    \begin{split}
    &\sup_{z\in\D}\left((1-|z|^2)\sum_{m=0}^\infty\left|\sum_{k=0}^m\left(\sum_{n=0}^\infty(n+1)\widehat{g}(n+1) \frac{z^{n+1}}{n+k+1}\right)(k+1)z^{m}\right|^2\int_R^1s^{2m+1}(1-s^2)\,ds\right)\\
    &\ge\sup_{0<r<1}\frac1{2\pi}\int_0^{2\pi}\Bigg((1-r)\sum_{m=0}^\infty r^{2m}\left(\int_R^1s^{2m+1}(1-s)\,ds\right)\\
    &\quad\cdot\left|
    \sum_{n=0}^\infty(n+1)\widehat{g}(n+1)\left(\sum_{k=0}^m\frac{(k+1)}{n+k+1}\right)
    (re^{i\t})^{n+1}\right|^2\Bigg)\,d\t\\
    &=\sup_{0<r<1}\Bigg((1-r)\sum_{m=0}^\infty r^{2m}\left(\int_R^1s^{2m+1}(1-s)\,ds\right)\\
    &\quad\cdot\sum_{n=0}^\infty(n+1)^2|\widehat{g}(n+1)|^2\left(\sum_{k=0}^m\frac{(k+1)}{n+k+1}\right)^2
    r^{2n+2}\Bigg)\\
    &\ge|\widehat{g}(N+1)|^2R^{2N+2}\left((1-R)\int_R^1\left(\sum_{m=0}^\infty (m+1)^2(Rs)^{2m}\right)s(1-s)\,ds
    \right)\\
    &\asymp|\widehat{g}(N+1)|^2R^{2N+2}\left((1-R)\int_R^1\frac{s(1-s)}{(1-(Rs)^2)^3}\,ds
    \right)\\
    &\gtrsim|\widehat{g}(N+1)|^2R^{2N+3},
    \end{split}
    \end{equation*}
so letting $R\to 1^-$, this contradicts \eqref{Eq:condition-ugly},
and the assertion follows.

(iii) The statements (a) and (c) are equivalent by Theorem~\ref{th:compactH2pairing}, and
clearly (c) implies (b).  If (b) is satisfied, then $T_g^*:A^*\to H^1$ is compact and hence (e) is satisfied. Also (b) clearly implies (d). Moreover, the proof of (i) shows that (e) implies (f). To complete the proof, it suffices to show that (a) follows by either (d) or (f).

Assume first (d). Then $T_g:\BMOA\to H^\infty$ is bounded by Theorem~\ref{Thm:main-1}(iv). Moreover, $T_g^{**}:\BMOA\to(H^\infty)^{**}$ is compact and since $T_g^{**}$ coincides with $T_g$ in $\VMOA^{**}\simeq\BMOA$ by Lemma~\ref{biadjunto_general}, $T_g:\BMOA\to H^\infty$ is compact.

Assume now (f), and let $\{f_n\}$ be a uniformly norm bounded family
in $\BMOA$ such that $f_n\to0$ uniformly on compact subsets of $\D$.
To see that $T_g:\BMOA\to H^\infty$ is compact, by
Lemma~\ref{le:weakdesc}, it suffices to show that
$\|T_g(f_n)\|_{H^\infty}\to0$. Let $\e>0$,
$C_1=\sup_{z\in\D}\|G^{H^2}_{g,z}\|_{H^{1}}<\infty$
and $N=N(R,\e)\in\N$ such that
$|f_n(0)|,|f'_n(\z)|<\e$ for all $n\ge N$ and $\z\in D(0,R)$. Then
H\"older's inequality yields
    \begin{equation*}
    \begin{split}
    \|T_g(f_n)\|_{H^\infty}
    &=\sup_{z\in\D}\left|\langle f_n,G^{H^2}_{g,z}\rangle_{H^2}\right|
    \lesssim\sup_{z\in\D}\int_\D|f_n'(\z)||(G^{H^2}_{g,z})'(\z)|\log\frac1{|\z|}\,dA(\z)+|f_n(0)|\|g\|_{H^\infty}\\
    &\lesssim C_1\e+\sup_{z\in\D}\int_{\D\setminus D(0,R)}|f_n'(\z)||(G^{H^2}_{g,z})'(\z)|(1-|\z|)\,dA(\z).
    \end{split}
    \end{equation*}
An application of the stopping time inequality
\eqref{Eq:stopping-time} to the function
$|f_n'|\cdot|(G^{H^2}_{g,z})'|\chi_{\D\setminus D(0,R)}$, together
with H\"older's inequality and \eqref{eq:BMOAnormequiv}, gives
    \begin{equation*}
    \begin{split}
    &\sup_{z\in\D}\int_{\D\setminus D(0,R)}|f_n'(\z)||(G^{H^2}_{g,z})'(\z)|(1-|\z|)\,dA(\z)
    \\ & \lesssim \sup_{z\in\D}
    \int_\T\left(\int_{\Gamma^{h(\xi)}(\xi)}|f_n'(u)|\cdot|(G^{H^2}_{g,z})'(u)|\chi_{\D\setminus D(0,R)}(u)\,dA(u)\right)|d\xi|\\
    &\lesssim\sup_{z\in\D}\int_\T C_{2,h}\left(f_n'(u)(1-|u|^2)\right)(\xi)A_{2,h}\left((G^{H^2}_{g,z})'(u)(1-|u|^2)\chi_{\D\setminus
    D(0,R)}(u)\right)(\xi)|d\xi|\\
    &\lesssim\|f_n\|_{\BMOA}\sup_{z\in\D}\int_{\T}\left(\int_{\Gamma(\xi)\setminus D(0,R)}\left|(G^{H^2}_{g,z})'(\z)\right|^2\,dA(\z)\right)^{\frac{1}{2}}|d\xi|,
    \end{split}
    \end{equation*}
and it follows that $T_g:\BMOA\to H^\infty$ is compact. Thus (a) is
satisfied and the proof if complete.
\end{Prf}

\medskip

Finally, we apply Theorem~\ref{th:compactH2pairing} to the
spaces $A^p_\om$ and $\B$.

\begin{theorem}
Let $\om\in\R$ and $g\in\H(\D)$.
\begin{itemize}
\item[\rm(i)] If $1<p<\infty$, then the following statements are equivalent:
\begin{enumerate}
\item[\rm(a)] $T_g:A^p_\omega\to H^\infty$ is compact;
\item[\rm(b)] $T_g:A^p_\omega\to A$ is compact;
\item[\rm(c)] $\{G^{A^2_\omega}_{g,z}:z\in \D\}$ is relatively compact in $A^{p'}_\omega$;
\item[\rm(d)]
    $\displaystyle
    \lim_{R\to1^-}\sup_{z\in\D}\int_{\D\setminus D(0,R)}\left|G^{A^2_\omega}_{g,z}(\z)\right|^{p'}\om(\z)\,dA(\z)=0.
    $
\end{enumerate}
\item[\rm(ii)]
The following statements are equivalent:
\begin{enumerate}
\item[(a)] $T_g:\B\to H^\infty$ is compact;
\item[(b)] $T_g:\B_0\to A$ is compact;
\item[(c)] $T_g:\B\to A$ is compact;
\item[(d)] $T_g:\B_0\to H^\infty$ is compact;
\item[(e)] $\{G^{A^2_\omega}_{g,z}:z\in\D\}$ is relatively compact in $A^1_\omega$;
\item[(f)]
    $
    \displaystyle
    \lim_{R\to1^-}\sup_{z\in\D}\int_{\D\setminus D(0,R)}\left|G^{A^2_\omega}_{g,z}(\z)\right|\om(\z)\,dA(\z)=0.
    $
\end{enumerate}
\end{itemize}
\end{theorem}

\begin{proof}
(i) The first three statements are equivalent by
Theorem~\ref{th:compactH2pairing} and \cite[Corollary~7]{PelRatproj}. If (c) is satisfied, then for
given $\e>0$, there exist $N=N(\e)\in\N$ and $h_1,\dots,h_N\in
A^{p'}_\om$ such that for each $z\in\D$, we have
$\|G^{A^2_\omega}_{g,z}-h_j\|_{A^{p'}_\om}<\e$ for some
$j=j(z)\in\{1,\ldots,N\}$. It follows that
    \begin{equation*}
    \begin{split}
    \left(\int_{\D\setminus D(0,R)}\left|G^{A^2_\omega}_{g,z}(\z)\right|^{p'}\om(\z)\,dA(\z)\right)^\frac1{p'}
    &\le\e+\left(\int_{\D\setminus D(0,R)}\left|h_j(\z)\right|^{p'}\om(\z)\,dA(\z)\right)^\frac1{p'},
    \end{split}
    \end{equation*}
and by choosing $R$ sufficiently close to 1, we deduce (d).

Assume now (d), and let $\{f_n\}$ be a uniformly norm bounded family
in $A^p_\om$ such that $f_n\to0$ uniformly on compact subsets of
$\D$. To see that $T_g:A^p_\omega\to H^\infty$ is compact, by
Lemma~\ref{le:weakdesc}, it suffices to show that
$\|T_g(f_n)\|_{H^\infty}\to0$. Let $\e>0$,
$C_1=\sup_n\|f_n\|_{A^p_\om}<\infty$, $R=R(\e)\in(0,1)$ such that
    $$
    \sup_{z\in\D}\int_{\D\setminus D(0,R)}\left|G^{A^2_\omega}_{g,z}(\z)\right|^{p'}\om(\z)\,dA(\z)<\e^{p'},
    $$
and $N=N(R,\e)\in\N$ such that $|f_n(\z)|<\e$ for all $n\ge N$ and
$\z\in D(0,R)$. Moreover, let
$C_2=\sup_{z\in\D}\|G^{A^2_\omega}_{g,z}\|_{A^{p'}_\om}^{p'}<\infty$.
Then, for $n\ge N$, \eqref{le:tgformula}, Fubini's theorem and
H\"older's inequality give
    \begin{equation*}
    \begin{split}
    \|T_g(f_n)\|_{H^\infty}
    &\le\sup_{z\in\D}\int_\D|f_n(\z)|\left|G^{A^2_\omega}_{g,z}(\z)\right|\om(\z)\,dA(\z)\\
    &\le\sup_{z\in\D}\int_{\overline{D(0,R)}}|f_n(\z)|\left|G^{A^2_\omega}_{g,z}(\z)\right|\om(\z)\,dA(\z)\\
    &\ +\sup_{z\in\D}\left(\int_{\D\setminus D(0,R)}|f_n(\z)|^p\om(\z)\,dA(\z)\right)^\frac1p\left(\int_{\D\setminus D(0,R)}\left|G^{A^2_\omega}_{g,z}(\z)\right|^{p'}\om(\z)\,dA(\z)\right)^\frac1{p'}\\
    &\le\e C_2+C_1\e=(C_1+C_2)\e,
    \end{split}
    \end{equation*}
and it follows that $T_g:A^p_\omega\to H^\infty$ is compact.

(ii) The statements (a) and (c) are equivalent by Theorem~\ref{th:compactH2pairing}, and
clearly (c) implies (b).  If (b) is satisfied, then $T_g^*:A^*\to A^1_\om$ is compact and hence (e) is satisfied. Also (b) clearly implies (d). Moreover, arguing as in the proof of (i), one sees that (e) implies (f). To complete the proof, it suffices to show that (a) follows by either (d) or (f).

Assume first (d). Then $T_g:\B\to H^\infty$ is bounded by Theorem~\ref{Thm-main-2}(iii). Moreover, $T_g^{**}:\B\to(H^\infty)^{**}$ is compact and since $T_g^{**}$ coincides with $T_g$ in $\B_0^{**}\simeq\B$ by Lemma~\ref{biadjunto_general}, $T_g:\B\to H^\infty$ is compact.

Assume now (f), and let $\{f_n\}$ be a uniformly norm bounded family
in $\B$ such that $f_n\to0$ uniformly on compact subsets of $\D$. To
see that $T_g:\B\to H^\infty$ is compact, by
Lemma~\ref{le:weakdesc}, it suffices to show that
$\|T_g(f_n)\|_{H^\infty}\to0$. Let $\e>0$,
$C_1=\sup_n\|f_n\|_{\B}<\infty$, and
$C_2=\sup_{z\in\D}\|G^{A^2_\omega}_{g,z}\|_{A^{1}_\om}<\infty$.
 Since $(1-r)M_1(r,f')\le4M_1(\frac{1+r}{2},f)$ for all
$f\in\H(\D)$ by the Cauchy formula and Fubini's theorem, a change of
variable, the hypothesis $\om\in\R$ and the assumption (f) imply
that there exists $R=R(\e)\in(0,1)$ such that
    $$
    \sup_{z\in\D}\int_{\D\setminus D(0,R)}\left|(G^{A^2_\omega}_{g,z})'(\z)\right|(1-|\z|)\om(\z)\,dA(\z)<\e.
    $$
Let $N=N(R,\e)\in\N$ such that $|f_n(0)|,|f'_n(\z)|<\e$ for all
$n\ge N$ and $\z\in D(0,R)$. We write
$\om^\star(z)=\int_{|z|}^1\log\frac{s}{|z|}\om(s)s\,ds$. Since
$\om\in\R$, $\om^\star(z)\asymp\om(z)(1-|z|)^2$ on $\D\setminus
D(0,R)$ \cite[(1.29)]{PelRat}. Then, for $n\ge N$,
\eqref{le:tgformula}, \cite[Theorem~4.2]{PelRat} and Fubini's
theorem give
    \begin{equation*}
    \begin{split}
    \|T_g(f_n)\|_{H^\infty}
    &\le4\sup_{z\in\D}\int_\D|f'_n(\z)|\left|(G^{A^2_\omega}_{g,z})'(\z)\right|\om^\star(\z)\,dA(\z)+|f_n(0)||G^{A^2_\omega}_{g,z}(0)|\\
    &\lesssim C_2\e+\|f_n\|_\B\sup_{z\in\D}\int_{\D\setminus
    D(0,R)}\left|(G^{A^2_\omega}_{g,z})'(\z)\right|\frac{\om^\star(\z)}{1-|\z|}\,dA(\z)\\
  &\lesssim(C_2+C_1)\e.
    \end{split}
    \end{equation*}
It follows that
$T_g:\B\to H^\infty$ is compact. Thus (a) is satisfied and the proof
is complete.
\end{proof}

\section{Further comments and questions}\label{section:comments}

By Theorem~\ref{th:BMOAcompactintro}, there exists $g\in\H(\D)$ such that $T_g:H^1\to H^\infty$ is bounded but not compact. Likewise  $T_g:H^\infty\to H^\infty$ can be bounded and not compact~\cite{AJS}. With regard to other admissible spaces $X$, excluding those for which $T(X,H^\infty)$ contains constant functions only, we do not know whether or not each bounded operator $T_g:X\to H^\infty$ is automatically compact. The answer might not be the same for all admissible spaces~$X$, and it seems natural to look at $X=H^2$ as a prototype for the Hardy spaces $H^p$, $1<p<\infty$, or certain weighted Bergman spaces $A^p_\om$. In order to shed some light on this question, we focus on analytic functions with nonnegative Taylor coefficients.
It is easy to see that such a function $g(z)=\sum_{n=0}^\infty
\widehat{g}(n)z^n$ belongs to $H^\infty$ if and only if
$\sum_n\widehat{g}(n)<\infty$. Therefore, by
\cite[Proposition~3.4]{AJS}, $T_g:H^\infty \to H^\infty$ is bounded
if and only if it is compact. The next result shows that the same
phenomenon holds on $H^2$.

\begin{theorem}\label{H2Hinftypositive}
Let $g(z)=\sum_{n=0}^\infty \widehat{g}(n)z^n\in\H(\D)$ such that
$\widehat{g}(n)\ge0$ for all $n$. Then the following statements
are equivalent:
\begin{itemize}
\item[\rm(i)] $T_g:H^2\to H^\infty$ is bounded;
\item[\rm(ii)]
$\sum_{k=0}^\infty \left(\sum_{n=0}^\infty
\frac{(n+1)\widehat{g}(n+1)}{n+k+1}\right)^2<\infty$;
\item[\rm(iii)] $T_g:H^2\to H^\infty$ is compact.
\end{itemize}
Moreover,
    \begin{equation}\label{H2Hinftynormpositive}
    \|T_g\|^2_{H^2\to H^\infty}\asymp \sum_{k=0}^\infty
    \left(\sum_{n=0}^\infty
    \frac{(n+1)\widehat{g}(n+1)}{n+k+1}\right)^2.
    \end{equation}
\end{theorem}

\begin{proof}
By  Theorem~\ref{Thm:main-1}, (i) is satisfied if and only if
$\sup_{z\in\D}\left\|G^{H^2}_{g,z}\right\|_{H^{2}}<\infty$. Since
    \begin{equation*}
    \begin{split}
    \overline{G^{H^2}_{g,z}(w)}=\int_0^z\frac{g'(\z)}{1-\overline{w}\z}\,d\z
    =\sum_{k=0}^\infty\left(\sum_{n=0}^\infty (n+1)\widehat{g}(n+1)\frac{z^{n+k+1}}{n+k+1}\right)\overline{w}^k,
    \end{split}
    \end{equation*}
we deduce
    \begin{equation}
    \begin{split}\label{H^2condition}
    \sup_{z\in\D}\left\|G^{H^2}_{g,z}\right\|_{H^{2}}=   \sup_{z\in\D}
    \left(\sum_{k=0}^\infty
   \left|\sum_{n=0}^\infty(n+1)\widehat{g}(n+1)\frac{z^{n+k+1}}{n+k+1}\right|^2
   \right)^{\frac{1}{2}}.
    \end{split}
    \end{equation}
Now, let us observe that for any $z\in\D$,
    $$
    \sum_{k=0}^\infty
   \left|\sum_{n=0}^\infty(n+1)\widehat{g}(n+1)\frac{z^{n+k+1}}{n+k+1}\right|^2\le
   \sum_{k=0}^\infty \left(\sum_{n=0}^\infty
    \frac{(n+1)\widehat{g}(n+1)}{n+k+1}\right)^2,
    $$
and hence \eqref{H^2condition} gives
    \begin{equation}
    \begin{split}\label{H^2condition2}
    \sup_{z\in\D}\left\|G^{H^2}_{g,z}\right\|^2_{H^{2}}\le
\sum_{k=0}^\infty \left(\sum_{n=0}^\infty
\frac{(n+1)\widehat{g}(n+1)}{n+k+1}\right)^2.
 \end{split}
    \end{equation}
Moreover, by Fatou's
lemma,
    \begin{equation*}
    \begin{split}
    \sum_{k=0}^\infty \left(\sum_{n=0}^\infty
    \frac{(n+1)\widehat{g}(n+1)}{n+k+1}\right)^2
    &\le\liminf_{\rho \to 1^-}\sum_{k=0}^\infty
    \left(\sum_{n=0}^\infty
    \frac{(n+1)\widehat{g}(n+1)}{n+k+1}\rho^n\right)^2 \rho^{2{k+1}}\\
    &\le\sup_{z\in\D}\sum_{k=0}^\infty\left|\sum_{n=0}^\infty(n+1)\widehat{g}(n+1)\frac{z^{n+k+1}}{n+k+1}\right|^2,
    \end{split}
    \end{equation*}
 which together with \eqref{H^2condition2} gives
    \begin{equation*}
    \begin{split}
    \sup_{z\in\D}\left\|G^{H^2}_{g,z}\right\|^2_{H^{2}}=
    \sup_{z\in\D}
    \sum_{k=0}^\infty
    \left|\sum_{n=0}^\infty(n+1)\widehat{g}(n+1)\frac{z^{n+k+1}}{n+k+1}\right|^2
    =\sum_{k=0}^\infty \left(\sum_{n=0}^\infty
    \frac{(n+1)\widehat{g}(n+1)}{n+k+1}\right)^2.
    \end{split}
    \end{equation*}
Therefore Theorem~\ref{Thm:main-1} shows that
(i) and (ii) are equivalent and \eqref{H2Hinftynormpositive} holds.

To complete the proof it remains to show that $T_g:H^2\to H^\infty$ is
compact if (ii) is satisfied. By Theorem~\ref{th:BMOAcompact}(i) and Fubini's theorem
$T_g:H^2\to H^\infty$ is compact if and only if
    \begin{equation}
    \begin{split}\label{H^2conditioncopact}
    0 &= \lim_{r\to 1^-} \sup_{z\in\D}\int_{\D\setminus D(0,r)}
    |{G^{H^2}_{g,z}}'(\z)|^2(1-|\z|)\,dA(\z)
    \\ & =  \lim_{r\to 1^-} \sup_{z\in\D}
    \sum_{k=1}^\infty
    \left|\sum_{n=0}^\infty(n+1)\widehat{g}(n+1)\frac{z^{n+k+1}}{n+k+1}\right|^2
    k^2\left(\int_r^1 s^{2k+1}(1-s)\,ds\right)
    \\ & = \lim_{r\to 1^-}
    \sum_{k=1}^\infty \left(\sum_{n=0}^\infty
    \frac{(n+1)\widehat{g}(n+1)}{n+k+1}\right)^2  k^2\left(\int_r^1
    s^{2k+1}(1-s)\,ds\right).
    \end{split}
    \end{equation}
For each $k\in\N$,
    $$
    \sup_{r\in (0,1)}k^2\left(\int_r^1
    s^{2k+1}(1-s)\,ds\right)\le k^2\left(\int_0^1
    s^{2k+1}(1-s)\,ds\right)\le 1
    $$
and $\lim_{r\to
    1^-}k^2\left(\int_r^1 s^{2k+1}(1-s)\,ds\right)=0$.
Hence, by applying (ii) and the
dominated convergence theorem we deduce
    \begin{equation*}
    \begin{split}
    &\lim_{r\to 1^-}
    \sum_{k=1}^\infty \left(\sum_{n=0}^\infty
    \frac{(n+1)\widehat{g}(n+1)}{n+k+1}\right)^2  k^2\left(\int_r^1
    s^{2k+1}(1-s)\,ds\right) \\ &= \sum_{k=1}^\infty
    \left(\sum_{n=0}^\infty \frac{(n+1)\widehat{g}(n+1)}{n+k+1}\right)^2
    \left(\lim_{r\to 1^-}k^2\left(\int_r^1
    s^{2k+1}(1-s)\,ds\right)\right) =0,
    \end{split}
    \end{equation*}
which together with \eqref{H^2conditioncopact} finishes the proof.
\end{proof}

It is also worth mentioning that if $g(z)=\sum_{n=0}^\infty
\widehat{g}(n)z^n\in \H(\D)$ has nonnegative Taylor coefficients, then
    $$
    \sum_{k=0}^\infty
    \left(\sum_{n=0}^\infty
    \frac{(n+1)\widehat{g}(n+1)}{n+k+1}\right)^2=\| H(g')\|^2_{H^2},
    $$
where $H(g')$ is the action of the operator induced by the classical
Hilbert matrix
$$H=\left(%
\begin{array}{ccccc}
            1 & \frac{1}{2}  & \frac{1}{3}  & .  \\
  \frac{1}{2} & \frac{1}{3}  & \frac{1}{4}  & .  \\
  \frac{1}{3} & \frac{1}{4}  & \frac{1}{5}  & . \\
  .  & . & . & .  \\
\end{array}%
\right)
$$
on $g'$.

\end{document}